\documentclass[11pt]{article}

\usepackage[utf8]{inputenc}
\pdfoutput=1
\usepackage[sc,osf]{mathpazo}
\usepackage[margin=1.1in,letterpaper]{geometry}
\usepackage[round]{natbib}
\makeatletter
\newlength\aftertitskip     \newlength\beforetitskip
\newlength\interauthorskip  \newlength\aftermaketitskip

\newcommand*{\mybox}[1]{\framebox{#1}}

\def\maketitle{\par
 \begingroup
   \def\thefootnote{\fnsymbol{footnote}}
   \def\@makefnmark{\hbox to 4pt{$^{\@thefnmark}$\hss}}
   \@maketitle \@thanks
 \endgroup
\setcounter{footnote}{0}
 \let\maketitle\relax \let\@maketitle\relax
 \gdef\@thanks{}\gdef\@author{}\gdef\@title{}\let\thanks\relax}

\def\@startauthor{\noindent \normalsize\bf}
\def\@endauthor{}
\def\@starteditor{\noindent \small {\bf Editor:~}}
\def\@endeditor{\normalsize}
\def\@maketitle{\vbox{\hsize\textwidth
 \linewidth\hsize \vskip \beforetitskip
 {\begin{center} \LARGE\@title \par \end{center}} \vskip \aftertitskip
 {\def\and{\unskip\enspace{\rm and}\enspace}%
  \def\addr{\small\it}%
  \def\email{\hfill\small\tt}%
  \def\name{\normalsize\bf}%
  \def\AND{\@endauthor\rm\hss \vskip \interauthorskip \@startauthor}
  \@startauthor \@author \@endauthor}
}}

\makeatother

\pdfoutput=1                    % force pdflatex to run

\usepackage{amsmath,amsthm}
\usepackage{graphicx}
\usepackage{color}
\usepackage{amssymb,amsfonts,amsxtra,mathrsfs,bm}
\usepackage{multicol}
\usepackage{multirow}
\usepackage{paralist}
\usepackage{xspace}
\usepackage{algorithm}% http://ctan.org/pkg/algorithm
\usepackage{algpseudocode}
\usepackage{hyperref}
\usepackage{bbm}
\usepackage{nicefrac}

\newcommand{\Mc}{\mathcal{M}}
\newcommand{\E}{\mathbb{E}}
\newcommand{\Xc}{\mathcal{X}}

\newcommand{\gfrak}{\mathfrak{g}}

\newcommand{\inv}[1]{#1^{-1}}

\newcommand{\nlsum}{\sum\nolimits}

\newcommand{\reals}{\mathbb{R}}

\newcommand{\posdef}{\mathbb{P}}
\newcommand{\pd}{\posdef}
\renewcommand{\H}{\mathbb{H}}

\newcommand{\ip}[2]{\langle {#1},\, {#2} \rangle}

\newcommand{\norm}[1]{\|{#1}\|}

\newcommand{\half}{\tfrac{1}{2}}

\newcommand{\set}[1]{\{ #1\}}

\DeclareMathOperator*{\argmin}{argmin}

\DeclareMathOperator{\sgn}{sgn}
\DeclareMathOperator{\trace}{tr}

\DeclareMathOperator{\Exp}{Exp}
\DeclareMathOperator{\grad}{grad}
\newcommand{\matlab}{\textsc{Matlab}\xspace}

\renewcommand{\H}{\mathbb{H}}

\newcommand{\fw}{\textsc{Fw}\xspace}
\newcommand{\rfw}{\textsc{Rfw}\xspace}
\newcommand{\srfw}{\textsc{Srfw}\xspace}
\newcommand{\svrfw}{\textsc{Svr-Rfw}\xspace}
\newcommand{\spiderfw}{\textsc{Spider-Rfw}\xspace}
\newcommand{\spider}{\textsc{Spider}\xspace}

\def\sint{\begingroup\textstyle\int\endgroup} % small integral

\usepackage{booktabs} % table
\usepackage{makecell} % table

\newtheorem{theorem}{Theorem}[section]
\newtheorem{lem}[theorem]{Lemma}

\newtheorem{cor}[theorem]{Corollary}

\theoremstyle{definition}
\newtheorem{defn}[theorem]{Definition}

\newtheorem{rmk}[theorem]{Remark}

\numberwithin{equation}{section}

\title{Projection-free nonconvex stochastic optimization \\ on Riemannian manifolds}
\author{\name Melanie Weber \email{mw25@math.princeton.edu}\\
  \addr{Princeton University}\\
  \name Suvrit Sra \email{suvrit@mit.edu}\\
  \addr{Laboratory for Information and Decision Systems, MIT}
}

\usepackage{amsopn}

\begin{document}
\maketitle

\begin{abstract}
We study stochastic projection-free methods for constrained optimization of smooth functions on Riemannian manifolds, i.e., with additional constraints beyond the parameter domain being a manifold. Specifically, we introduce stochastic Riemannian Frank-Wolfe methods for nonconvex and geodesically convex problems. We present algorithms for both purely stochastic optimization and finite-sum problems. For the latter, we develop variance-reduced methods, including a Riemannian adaptation of the recently proposed \spider technique. For all settings, we recover convergence rates that are comparable to the best-known rates for their Euclidean counterparts. Finally, we discuss applications to two classic tasks: The computation of the Karcher mean of positive definite matrices and Wasserstein barycenters for multivariate normal distributions. For both tasks, stochastic \fw methods yield state-of-the-art empirical performance.
\end{abstract}

%%%%%%%%%%%%
\section{Introduction}
We study the following constrained (and possibly nonconvex) stochastic and finite-sum problems:
\begin{align}
  \label{eq:opt1}
  \min_{x \in \Xc \subset \mathcal{M}} \Phi(x) &:= \mathbb{E}_\xi[\phi (x,\xi)] = \sint \phi (x,\xi) dP(\xi),\\
  \label{eq:opt2}
 \min_{x \in \Xc \subset \mathcal{M}} \Phi(x) &:= \tfrac{1}{m}\nlsum_{i=1}^m\phi_i(x),
\end{align}
where $\Xc$ is compact and geodesically convex and $\mathcal{M}$ is a Riemannian manifold. Moreover, the component functions $\set{\phi_i}_{i=1}^m$ as well as $\Phi$ are (geodesically) Lipschitz-smooth, but may be nonconvex. These problems greatly generalize their Euclidean counterparts (where $\Mc\equiv \reals^d$), which themselves are of central importance in optimization and machine learning. 
In particular, finite-sum problems (Eq.~\ref{eq:opt2}) arise frequently in machine learning subroutines, such as Empirical Risk Minimization, Maximum likelihood estimation or the computation of M-estimators. 

There has been an increasing interest in solving Riemannian problems of the above form, albeit without constraints~\citep{bonnabel,rspider1,zhang16,zhang_colt16,kasai-mishra,kasai-mishra2,tripuraneni2018averaging}. 
This interest is driven by two key motivations: First, that the exploitation of Riemannian geometry can deliver algorithms that are computationally superior to standard nonlinear programming approaches~\citep{absil_book,udriste1994convex,zhang16,manopt}. Secondly, in many applications we encounter non-Euclidean data, such as graphs, strings, matrices, tensors; where using a forced Euclidean representation can be quite inefficient~\citep{sala,nickel2017poincare,aistats,zhang16,billera2001geometry,edelman1998geometry}. 
These motivations have driven the recent surge of interest in the adaption and generalization of machine learning models and algorithms to Riemannian manifolds.

We solve problem~(\ref{eq:opt1}) by introducing Riemannian stochastic Frank-Wolfe (\fw) algorithms. These methods are projection-free~\citep{fw}, a property that has driven much of the recent interest in them~\citep{jaggi2013revisiting}. In contrast to projection-based methods, the \fw update requires solving a ``linear" optimization problem that ensures feasibility while often being much faster than projection. \fw has been intensively studied in Euclidean spaces for both convex~\citep{julien15,jaggi2013revisiting} and nonconvex~\citep{LJ16} 
objectives. Furthermore, stochastic variants have been proposed~\citep{reddi1} that enable strong performance gains. As our experiments will show, our stochastic Riemannian \fw also delivers similarly strong performance gains on sample applications, outperforming the state-of-the-art.

\subsection{Summary of main contributions}
%\begin{list}{\tiny$\blacksquare$}{\leftmargin=1.3em}
%\setlength{\itemsep}{1pt}
%\vspace*{-6pt}
\begin{itemize}
\item We introduce three algorithms: (i) \srfw, a fully stochastic method that solves~(\ref{eq:opt1}); (ii) \svrfw, a semi-stochastic variance-reduced version for~(\ref{eq:opt2}); and (iii) \spiderfw, an improved variance-reduced variant that uses the recently proposed \spider technique for estimating the gradient.  All three algorithms generalize various stochastic gradient tools to the Riemannian setting. For all methods, we establish convergence rates to first-order stationary points that match the rates of their Euclidean counterparts. Under the stronger assumption of geodesically convex objectives, we recover global sublinear convergence rates.
\item In contrast to~\citep{weber-sra}, which consider Riemannian \fw, \textsc{Stochastic} \rfw does not require the computation of full gradients. Overcoming the need to compute the full gradient in each iteration greatly reduces the computational cost of each iteration as it removes a major bottleneck in \rfw. Moreover, \textsc{Stochastic} \rfw applies to problem~\ref{eq:opt1}, a crucial subroutine in many machine learning applications.
\item We present an application to the computation of Riemannian centroids (Karcher mean) for positive definite matrices. This task is a well-known benchmark for Riemannian optimization, and it arises, for instance, in statistical analysis, signal processing and computer vision.  Notably, a simpler version of it also arises in the computation of hyperbolic embeddings.
\item Furthermore, we present an application to the computation of Wasserstein barycenters for multivariate and \emph{matrix-variate} Gaussians. For the latter, we prove the somewhat surprising property that the Wasserstein distance between two matrix-variate Gaussians is Euclidean convex. This result may be of independent interest.
\end{itemize}
%\end{list}
The proposed \textsc{Stochastic} \rfw methods deliver valuable improvements, both in theory and experiment. Table~\ref{tab:complexities} summarizes the complexity results for all variants in comparison with \rfw (Algorithm~\ref{alg.rfw}). For an analysis of \rfw's complexity, see~\citep[Theorem~3]{weber-sra}.
Our algorithms outperform state-of-the art batch methods such as Riemannian \textsc{LBFGS}~\citep{lbfgs} and \textsc{Zhang}'s majorization-minimization algorithm~\citep{zhang}. Moreover, we also observe performance gains over the deterministic \rfw, which itself is known to be competitive against a wide range of Riemannian optimization tools~\citep{weber-sra}. Importantly, our methods further outperform state-of-the-art stochastic Riemannian methods \textsc{RSG}~\citep{kasai-mishra} and \textsc{RSVRG}~\citep{sato2017riemannian,zhang16}. 

\subsection{Related work}
Riemannian optimization has recently witnessed a surge of interest~\citep{bonnabel,zhang_colt16,absil_rbfgs,boumal}. A comprehensive introduction to Riemannian optimization can be found in~\citep{absil_book}. The \textsc{Manopt} toolbox~\citep{manopt} implements many successful Riemannian optimization methods, serving as a benchmark.

The study of stochastic methods for Riemannian optimization has largely focused on projected-gradient methods. \citet{bonnabel} introduced the first Riemannian \textsc{SGD}; \citet{zhang_colt16} present a systematic study of first-order methods for geodesically convex problems, followed by a variance-reduced Riemannian \textsc{SVRG}~\citep{zhang16, sato2017riemannian} that also applies to  geodesically nonconvex functions.  \citet{kasai-mishra} study gradient descent variants, as well as a Riemannian \textsc{ADAM}~\citep{kasai-mishra2}. A caveat of these methods is that a potentially costly projection is needed to ensure convergence. Otherwise, the strong (and often unrealistic) assumption that their iterates remain in a compact set is required 
In contrast, \rfw~(Algorithm~\ref{alg.rfw}) generates feasible iterates directly and therefore avoids the need to compute projections.  This leads to a cleaner analysis and a more practical method in cases where the ``linear'' oracle is efficiently implementable~\citep{weber-sra}. We provide additional details on the comparison of projection-free and projection-based methods in section~\ref{sec:proj-free}.
Riemannian optimization has also been applied in the ML literature, including for the computation of hyperbolic embeddings~\citep{sala}, low-rank matrix and tensor factorization~\citep{vandereycken} and eigenvector based methods~\citep{journee,zhang16,tripuraneni2018averaging}.

\begin{algorithm}[ht]
  \caption{Riemannian Frank-Wolfe (\rfw)}
  \label{alg.rfw}
  \begin{algorithmic}[1] 
    \State Initialize $x_0 \in \Xc \subseteq \mathcal{M}$; assume access to the geodesic map $\gamma: [0,1] \rightarrow \mathcal{M}$
     \For {$k=0,1,\dots$}
        \State $z_k \gets \argmin_{z \in \Xc}\ \ip{{\rm grad} \; \phi(x_k)}{\Exp_{x_k}^{-1}(z)}$
        \State Let $\eta_k \gets \frac{2}{k+2}$
        \State $x_{k+1} \gets \gamma(\eta_k)$, where $\gamma(0)=x_k$ and $\gamma(1)=z_k$
     \EndFor
   \end{algorithmic}
\end{algorithm}
{\renewcommand{\arraystretch}{2}
\begin{table*}[t]
\begin{small}
\label{tab:complexities}
\noindent\makebox[\textwidth]{
    %\begin{tabular}{@{} lcc @{}}
    \begin{tabular}{lcccc}
      \toprule
      \thead{Algorithm} & \thead{\rfw} & \thead{\srfw} & \thead{\svrfw} & \thead{\spiderfw} \\
      \midrule
      \textsc{SFO/ IFO} & $O\left(	\frac{m}{ \epsilon^2	}\right)$ &$O\left(\frac{1}{\epsilon^4}\right)$ & $O\left(m + \frac{m^{2/3}}{\epsilon^2} \right)$ & $O \left(\frac{1}{\epsilon^3}\right)$ \\
      \textsc{RLO} & $O \left(	\frac{1}{\epsilon^2}	\right)$ & $O\left(\frac{1}{\epsilon^2}\right)$ & $O \left(\frac{1}{\epsilon^2}\right)$ & $O \left(\frac{1}{\epsilon^2}\right)$ \\
       \bottomrule
    \end{tabular}
    }
    \caption{Oracle complexities of our Stochastic Riemannian Frank-Wolfe methods versus \rfw~\citep{weber-sra} for nonconvex objectives. Note that we recover the best known rates of the Euclidean counterparts for each method.  We consider three different oracle models, which we will define below in section~\ref{sec:complexities}: \textbf{SFO/ IFO:} Stochastic first-order oracle (for stochastic objectives) and Incremental first-order oracle (for objectives with finite-sum form).  \textbf{LO:} Riemannian linear optimization oracle.} 
  \end{small}
  \end{table*}
}

%%%%%%%%%%%%%%%%%%
\section{Background and Notation}
\vspace*{-5pt}
We start by recalling some basic background on Riemannian geometry and introduce necessary notation. For a comprehensive overview on Riemannian geometry, see, e.g., ~\citep{jost}.

\subsection{Riemannian manifolds} 
A \emph{manifold} $\mathcal{M}$ is a locally Euclidean space equipped with a differential structure. Its corresponding \emph{tangent spaces} $T_x \mathcal{M}$ consist of tangent vectors at points $x \in \mathcal{M}$. We define an \emph{exponential map} $\Exp: T_x \mathcal{M} \rightarrow \mathcal{M}$ as follows: Let $g_x \in T_x\mathcal{M}$; then $y=\Exp_x(g_x) \in \mathcal{M}$ with respect to a geodesic $\gamma: [0,1] \mapsto \mathcal{M}$ with $\gamma(0)= x$, $\gamma(1)=y$ and $\dot{\gamma}(0)=g_x$. We will also use the \emph{inverse} exponential map 
$\Exp^{-1}: \mathcal{M} \rightarrow T_x \mathcal{M}$
that defines a diffeomorphism from the neighborhood of $x \in \mathcal{M}$ onto the neighborhood of $0 \in T_x \mathcal{M}$ with $\Exp_x^{-1}(x)=0$. 

\emph{Riemannian manifolds} are smooth manifolds with an inner product $\gfrak_x(u,v)=\ip{u}{v}_x$ defined on $T_x \mathcal{M}$ for each $x\in \mathcal{M}$. The inner product gives rise to a norm $\| v \|_x := \sqrt{\gfrak_x(v,v)}$ for $v \in T_x \mathcal{M}$. We will further denote the geodesic distance of $x,y \in \mathcal{M}$ as $d(x,y)$. For comparing vectors of different tangent spaces, we use the following notion of \emph{parallel transport}: Let $x, y \in \mathcal{M}$, $x \neq y$. Then, the operator $\Gamma_x^y g_x$ maps $g_x \in T_x \mathcal{M}$  to the tangent space $T_y \mathcal{M}$ along a geodesic $\gamma$ with $\gamma(0)=x$ and $\gamma(1)=y$. Note that the inner product on the tangent spaces is preserved under this mapping.

\subsection{Gradients, smoothness and convexity} 
The \emph{Riemannian gradient} $\grad \; \phi(x)$ of a differentiable function $\phi : \mathcal{M} \rightarrow \mathbb{R}$ is defined as the unique vector in $T_x \mathcal{M}$ with directional derivative $D\phi(x)[v] = \ip{\grad \; \phi(x)}{v}_x$ for all $v \in T_x \mathcal{M}$. For our algorithms we further need a notion of smoothness: Let $\phi : \mathcal{M} \rightarrow \mathbb{R}$ be differentiable. We say that $\phi$ is \emph{$L$-smooth}, if
\begin{equation}
  \label{eq:L}
  \|\grad \; \phi(y)-\Gamma_x^y\grad \; \phi(x)\| \le L d(x,y), \; \forall\ x, y \in  \mathcal{M},
\end{equation}
or equivalently, if for all $x, y \in \Mc$, $\phi$ satisfies
\begin{equation} \label{eq:L2}
  \phi(y) \leq \phi(x) + \ip{\grad \; \phi(x)}{\Exp_x^{-1} (y)}_x + \tfrac{L}{2} d^2(x,y).
\end{equation}
Another important property is \emph{geodesic convexity} (short: g-convexity), which is defined as 
\begin{equation}\label{eq:g-convex}
\phi(y) \geq \phi (x) + \ip{\grad \; \phi(x)}{\Exp_x ^{-1}(y)}_x \; \forall x,y \in \mathcal{M} \; .
\end{equation}

\subsection{Projection-free vs.\ Projection-based methods.}\label{sec:proj-free}
Classic Riemannian optimization has focused mostly on projection-based methods, such as \emph{Riemannian Gradient Descent} (\textsc{RGD}) or \emph{Riemannian Steepest Descent} (\textsc{RSD})~\citep{absil_book}. A convergence analysis of such methods typically assumes the gradient to be Lipschitz.  However, the objectives typically considered in most optimization and machine learning tasks are not Lipschitz on the whole manifold. Hence, a compactness condition is required. Crucially, in projection-based methods, the retraction back onto the manifold is typically not guaranteed to land in this compact set. Therefore, additional work (e.g., a projection step) is needed to ensure that the update remains in the compact region where the gradient is Lipschitz. On the other hand, \fw methods bypass this issue, because their update is guaranteed to stay within the compact feasible region. Further, for descent based methods it can suffice to ensure boundedness of the initial level set, but crucially, stochastic methods are \emph{not} descent methods, and this argument does not apply. Finally, in some problems, the Riemannian ``linear" oracle can be much less expensive than computing a projection back onto the compact set. This is particularly significant for the applications highlighted in this paper, where the ``linear" oracle can even be solved in closed form.

\subsection{Oracle models}\label{sec:complexities}
We briefly review three oracle models, which are commonly used to understand the complexity of stochastic optimization algorithms. 
\begin{enumerate}
\item \emph{Stochastic First-order Oracle} (short: \emph{SFO}): Consider a stochastic function $\Phi(x):= \mathbb{E} \left[ \phi(x, \xi)	\right]$ with $\xi \sim \mathcal{P}$.  For an input $x \in \mathcal{M}$, the SFO returns $(\phi(x,\xi'), \nabla \phi(x,\xi'))$ for a sample $\xi'$  that is drawn i.i.d. from the distribution $\mathcal{P}$. For details, see~\citep{nemirovski}.
\item \emph{Incremental First-order Oracle} (short: \emph{IFO}):
Consider a finite sum $\Phi(x):=\frac{1}{m} \sum_{i} \phi_i (x)$.  For an input $(i,x)$, where $i \in [n]$ is a function index and $x \in \mathcal{M}$, the IFO returns $(\phi_i(x), \nabla \phi_i(x))$. For details, see~\citep{agarwal_bottou}.
\item \emph{Riemannian Linear Optimization Oracle} (short: \emph{RLO}): For a set of constraints $\Xc$, a point $x \in \Xc \subseteq \mathcal{M}$ and a direction $g \in T_x \mathcal{M}$, the RLO returns $\argmin_{z \in \Xc} \ip{g}{\Exp_x^{-1}(z)}$.
\end{enumerate}
Throughout the paper, we measure complexity as the number of SFO/ IFO and RLO calls made by the algorithm to obtain an $\epsilon$-accurate solution. 

%%%%%%%%%%%%%%%%%%
\section{Algorithms}
In this section, we introduce three stochastic variants of \rfw and analyze their convergence. % and provide a full non-asymptotic convergence 
Here and in the following $x_k, x_{k+1}$ and $y$ are as specified in Algorithm~\ref{alg.srfw}, ~\ref{alg.svrfw} and~\ref{alg.spider} respectively.
We further make the following assumptions: (1) $\Phi$ is $L$-smooth; and (2) in the stochastic case, the norm of the stochastic gradient  is bounded as
\begin{align*}
\max_{\substack{x \in \Xc \\ \xi \in {\rm supp}(\mathcal{P})}} \norm{\grad \phi(x,\xi)} \leq C
\end{align*}
for some constant $C \geq 0$.

\subsection{Stochastic Riemannian Frank-Wolfe}
\begin{algorithm*}[t]
  \caption{Stochastic Riemannian Frank-Wolfe (\srfw)}
  \label{alg.srfw}
  \begin{algorithmic}[1] 
    \State Initialize $x_0 \in \Xc$,  assume access to the geodesic map $\gamma: [0,1] \rightarrow \mathcal{M}$. 
    \State Set number of iterations $K$ and minibatch sizes $\lbrace b_k \rbrace_{k=0}^{K-1}$.
     \For {$k=0,1,\dots K-1$}
        \State Sample i.i.d. $\lbrace \xi_{1} , ..., \xi_{b_k} \rbrace$ uniformly at random according to $\mathcal{P}$.
        \State $y_k \gets \argmin_{y \in \Xc}\ \ip{\frac{1}{b_k}\sum_{i=1}^{b_k} \grad \; \phi (x_k, \xi_i)}{\Exp_{x_k}^{-1}(y)}$
        \State Compute step size $\eta_k$ and set $x_{k+1} \gets \gamma(\eta_k)$, where $\gamma(0)=x_k$ and $\gamma(1)=y_k$.
        \State $x^k \gets x_k$
     \EndFor
     \State Output $\hat{x}$ chosen uniformly at random from $\lbrace x^k	\rbrace_{k=0}^{K-1}$.
   \end{algorithmic}
\end{algorithm*}
\noindent Our first method, \srfw (Algorithm~\ref{alg.srfw}), is a direct analog of stochastic Euclidean \fw. It has two key computational components: A stochastic gradient and a ``linear'' oracle. Specifically, it requires access to the \emph{stochastic} ``linear'' oracle
\begin{equation}
  \label{eq:1}
y_k \gets \argmin_{y \in \Xc}\ \ip{G(\xi, x_k)}{\Exp_{x_k}^{-1}(y)}\; ,
\end{equation}
where $G(\cdot,\cdot)$ is an unbiased estimator of the Riemannian gradient ($\mathbb{E}_\xi G(\xi,x) = {\rm grad} \; \Phi(x)$). In contrast to Euclidean \fw,  the oracle~\eqref{eq:1} involves solving a nonlinear, nonconvex optimization problem. Whenever this problem is efficiently solvable, we can benefit from the FW strategy. In~\citep{weber-sra}, we analyze two instances where Eq.~\ref{eq:1} can be solved in closed form, for positive definite matrices and for the special orthogonal group respectively. Our experiments below will provide two concrete examples for the case of positive definite matrices.

We consider a minibatch variant of the oracle~\eqref{eq:1}, namely
\begin{align*}
y_k \gets  \argmin_{y \in \Xc}\ \Bigl\langle\frac{1}{b_k}\nlsum_{i=1}^{b_k} \grad \; \phi (x_k, \xi_i), \Exp_{x_k}^{-1}(y)\Bigr\rangle,
\end{align*}
where $\xi_i \sim \mathcal{P}$ are drawn i.i.d., and thus the minibatch gradient is also unbiased. 
We first evaluate the goodness of this minibatch gradient approximation with the following (standard) lemma:\\
\begin{lem}[Goodness of stochastic gradient estimate]
\label{lem:RM}
Let $\Phi(x)=\mathbb{E}_{\bm \xi} \left[ \phi(x,\xi_i) \right]$ with random variables $\lbrace	\xi_i \rbrace_{i=1}^b = \bm{\xi} \sim \mathcal{P}$. Furthermore, let
$g (x) := \frac{1}{b} \sum_{i=1}^{b} \grad \; \phi(x, \xi_i)$ denote the gradient estimate from a batch $\bm{\xi}$. Assume that the norm of the gradient estimate is upper-bounded as 
$\max_{x \in \Xc, \xi \in {\rm supp}(\mathcal{P}) } \norm{\grad \phi(x,\xi)} \leq C$. Then, $\mathbb{E}_{\bm{\xi}} \left[ \norm{g (x) - \grad \; \Phi(x)}	\right] \leq \frac{C}{\sqrt{b}}$. 
\end{lem}
\noindent In the following, we drop the subscript $\bm{\xi}$ from the expectation for ease of notation, whenever its meaning is clear from context.

For the proof, recall the following fact, which we will use throughout the paper:
\begin{rmk}\normalfont
\label{eq:xi-sum-squares}
For a set of $n$ independent random variables $\lbrace \nu_i \rbrace_{1 \leq i \leq n}$ with mean zero, we have
\begin{equation}
\mathbb{E} \left[\norm{\nu_1 + \dots + \nu_n }^2	\right] = \mathbb{E} \left[	\norm{\nu_1}^2 + \dots + \norm{\nu_n}^2	\right] \; .
\end{equation}
\end{rmk}
\begin{proof}
We have
\begin{align*}
\mathbb{E} \left[\norm{	g(x) - \underbrace{\grad \; \Phi(x)}_{=\mathbb{E}\left[	g(x)	\right]}}^2	\right] 
&= \mathbb{E} \left[\norm{	g(x)}^2 \right] - \underbrace{\norm{\mathbb{E} \left[	g(x)\right]}^2}_{\geq 0}
\leq \mathbb{E} \left[ \norm{ g(x)}^2	\right] \\
&= \mathbb{E} \left[ \norm{  \frac{1}{b} \sum_{i=1}^b \grad \phi(x, \xi_i)	}^2	\right]
\overset{(1)}{\leq} \frac{1}{b^2} \mathbb{E} \left[	\sum_{i=1}^b \underbrace{\norm{\grad \phi (x,\xi_i)}^2}_{\leq C^2}	\right] \overset{(2)}{\leq} \frac{C^2}{b} \; ,
\end{align*}
where (1) follows from Remark~\ref{eq:xi-sum-squares} and the fact that $\mathbb{E}  \left[	g(x) - \grad \Phi(x)	\right]=0$, since $g(x)$ is assumed to be an unbiased gradient estimate.  (2) from the assumption that the norm of the gradient is upper-bounded by $C$. Furthermore, with Jensen's inequality:
\begin{align*}
\mathbb{E}\left[	\norm{g(x) - \grad \; \Phi(x)}^2	\right] \geq \left[\mathbb{E} \left(	\norm{g(x) - \grad \; \Phi(x)}	\right)	\right]^2 \; .
\end{align*}
Putting both together and taking the square root on both sides gives the desired claim:
\begin{align*}
\mathbb{E} \left[	\norm{g(x) - \grad \; \Phi(x)}	\right] \leq \frac{C}{\sqrt{b}} \; .
\end{align*}
\end{proof}

\noindent With this characterization of the approximation error, we can perform a convergence analysis for both nonconvex and g-convex objectives. To evaluate convergence rates, consider the following criterion (\emph{Frank-Wolfe gap}):
\begin{align}\label{def:FW-gap}
\mathcal{G}(x) = \max_{y \in \mathcal{X}} \ip{\Exp_{x}^{-1}(y)}{-\grad \Phi(x)} \; .
\end{align}
\noindent A similar criterion is used in theoretical analysis of Euclidean Frank-Wolfe methods (see, e.g.,~\citet{reddi1}). 
We define the \emph{Stochastic Frank-Wolfe gap} as
\begin{align*}
\hat{\mathcal{G}}(x)  = \max_{y \in \mathcal{X}} \ip{\Exp_{x}^{-1}(y)}{-g(x)} \; .
\end{align*}
Assuming that the Robbins-Monroe approximation $g(x)$ gives an unbiased estimate of the gradient $\grad \Phi(x)$ (*), we have (by Jensen's inequality and the convexity of the \emph{max}-function):
\begin{align*}
\mathbb{E} \left[ \hat{\mathcal{G}}(x)	\right]
\geq \max_{y \in \mathcal{X}} \ip{\Exp_{x}^{-1}(y)}{-\mathbb{E} \left[g(x) \right]}
\overset{(*)}{=} \max_{y \in \mathcal{X}} \ip{\Exp_{x}^{-1}(y)}{-\grad \Phi (x)} = \mathcal{G}(x) \; .
\end{align*}
With this, we can show that \srfw converges at a sublinear rate to first-order stationary points:
\begin{theorem}[Convergence \srfw]
\label{thm:srfw}
With constant steps size $\eta_k = \frac{1}{\sqrt{K}}$ and constant batch sizes $b_k=K$, Algorithm~\ref{alg.srfw} converges in expectation with a \emph{sublinear} rate, i.e.
\begin{align*}
\mathbb{E}_{\bm{\xi}} \left[ \mathcal{G}(\hat{x}) \right]=O(1/\sqrt{K}) \; .
\end{align*}
\end{theorem}
\noindent To prove the theorem, we need a few additional auxiliary results. First, recall the definition of the \emph{curvature constant} $M_{\Phi}$, introduced in~\citep{weber-sra}:
\begin{defn}[Curvature constant]
Let $x,y,z \in  \Xc$ and $\gamma: [0,1] \rightarrow \mathcal{M}$ a geodesic map with $\gamma(0)=x$, $\gamma(1)=z$ and $y=\gamma(\eta)$ for $\eta \in [0,1]$. Define
\begin{equation}
  \label{eq:defm}
  M_\Phi := \sup_{\substack{x,y,z \in \Xc \\ y = \gamma(\eta)}} 
  \tfrac{2}{\eta^2}\left[\Phi(y) - \Phi(x) - \ip{\grad \Phi(x)}{\Exp_x ^{-1}(y)} \right] \; .
\end{equation}
\end{defn}
\noindent We further recall two technical lemmas on $M_{\Phi}$; the proofs can be found in~\citep{weber-sra}:
\begin{lem}[\cite{weber-sra}]
\label{lem:bound-on-M}
Let $\Phi: \mathcal{M} \to \reals$ be $L$-smooth on $\Xc$; let $\mathrm{diam}(\Xc) := \displaystyle\sup_{x,y \in \Xc}{\rm d}{(x,y)}$. Then, % With 
the curvature constant $M_\phi$ satisfies the bound $M_{\Phi} \leq L \; {\rm diam}(\Xc)^2$.
\end{lem}
\begin{lem}[\cite{weber-sra}]
  \label{A.smooth}
  Let $\Xc$ be a constrained set. There exists a constant $M_{\Phi} \ge 0$ such that for $x_k, x_{k+1}, y_k \in \Xc$ as specified in Algorithm~\ref{alg.srfw}, and for $\eta_k \in (0,1)$
 \begin{equation*}
    \Phi(x_{k+1}) \le \Phi(x_k) + \eta_k \ip{\grad \; \Phi(x_k)}{\Exp_{x_k} ^{-1}(y_k)} + \half M_{\Phi}\eta_k^2.
  \end{equation*}
\end{lem}
\noindent With this, we can now prove Theorem~\ref{thm:srfw}:\\
\begin{proof}(Theorem~\ref{thm:srfw})
Let again
\begin{align}
g_k (x_k) := \frac{1}{b_k} \sum_{i=1}^{b_k} \grad \; \phi(x_k, \xi_i)
\end{align}
denote the gradient estimate from the $k^{th}$ batch. Then
\begin{align}\label{eq:srfw-1}
\Phi(x_{k+1}) &\overset{(1)}{\leq} \Phi(x_k) + \eta_k \ip{\grad \; \Phi(x_k)}{\Exp_{x_k}^{-1}(y_k)} + \frac{1}{2} M_{\Phi} \eta_k^2 \\
&\overset{(2)}{\leq} \Phi(x_k) + \eta_k \ip{g_k(x_k)}{\Exp_{x_k}^{-1}(y_k)} + \eta_k \ip{\grad \; \Phi(x_k) - g_k(x_k)}{\Exp_{x_k}^{-1}(y_k)} + \frac{1}{2} M_{\Phi} \eta_k^2 \; 
\end{align}
Here, (1) follows from Lemma~\ref{A.smooth} and (2) from 'adding a zero' with respect to $g_k$. 
We then apply the Cauchy-Schwartz inequality to the inner product and make use of the fact that the geodesic distance between points in $\Xc$ is bounded by its diameter:
\begin{align}
\ip{\grad \; \Phi(x_k) - g_k(x_k)}{\Exp_{x_k}^{-1}(y_k)} \leq \norm{\grad \; \Phi(x_k) - g_k(x_k)} \cdot \underbrace{\norm{\Exp_{x_k}^{-1}(y_k)}}_{\leq {\rm diam}(\Xc)} \; .
\end{align}
This gives (with $D := {\rm diam}(\Xc)$)
\begin{align*}
\Phi(x_{k+1}) &\leq \Phi(x_k) + \eta_k \ip{g_k(x_k)}{\Exp_{x_k}^{-1}(y_k)} + \eta_k D \norm{\grad \; \Phi(x_k) - g_k(x_k)} + \frac{1}{2} M_{\Phi} \eta_k^2 \; .
\end{align*}
%Let $\Delta_k := \Phi(x_{k}) - \Phi(x^*)$ denote the optimality gap. Subtracting $\Phi(x^*)$ from both sides, we can rewrite the above inequality as
%\begin{align*}
%\Delta_{k+1} &\leq \Delta_k + \eta_k \ip{g_k(x_k)}{\Exp_{x_k}^{-1}(y_k)}  + \eta_k D \norm{\grad \; \Phi(x_k) - g_k(x_k)} + \frac{1}{2} M_{\Phi} \eta_k^2 \; .
%\end{align*}
Taking expectations and applying Lemma~\ref{lem:RM} to the third term on the right-hand-side, we get 
\begin{align*}
\mathbb{E}\left[\Phi(x_{k+1)} \right] &\leq \mathbb{E}\left[\Phi(x_k) \right] - \eta_k \mathbb{E}\left[\hat{\mathcal{G}} (x_k) \right] + \eta_k D \frac{C}{\sqrt{b_k}} + \frac{1}{2} M_{\Phi} \eta_k^2 \; ,
\end{align*}
where we have rewritten the second term in terms of the stochastic Frank-Wolfe gap
\begin{align*}
\mathbb{E} \left[\hat{\mathcal{G}}(x_k) \right] =  - \mathbb{E} \left[  \ip{g_k(x_k)}{\Exp_{x_k}^{-1}(y_k)} \right] \; .
\end{align*}
%The equality follows from $y_k$ being optimal w.r.t. the oracle as defined in Alg.~\ref{alg.srfw}. 
Summing over all $k$ batches, telescoping and reordering terms gives
\begin{align}\label{eq:ineq}
\sum_k \eta_k \mathbb{E}\left[\hat{\mathcal{G}} (x_k) \right] &\leq \mathbb{E}\left[\Phi(x_{0}) \right] - \mathbb{E}\left[\Phi(x_{K}) \right] + \sum_k \eta_k D \frac{C}{\sqrt{b_k}} + \sum_k \frac{1}{2} M_{\Phi} \eta_k^2 \\
&\leq \left(\Phi(x_{0}) - \Phi(x_K) \right) + \sum_k \eta_k D \frac{C}{\sqrt{b_k}} + \sum_k \frac{1}{2} M_{\Phi} \eta_k^2 \; .
\end{align}
From Algorithm~\ref{alg.srfw} we see that the output $\hat{x}$ is chosen uniformly at random from $\lbrace x_1 , ... , x_K \rbrace$, i.e. $\mathbb{E} \left[\mathbb{E}\left[\hat{\mathcal{G}} (x_k) \right] \right]= \mathbb{E}\left[\mathcal{G} (\hat{x}) \right]$, where we have used that, by construction, $\mathbb{E}\left[ \hat{\mathcal{G}}(x) \right] = \mathcal{G}(x)$. Now, with constant step sizes $\eta_k=\eta$ and batch sizes $b_k=b$, we have
\begin{align*}
K \eta \mathbb{E}\left[\mathcal{G} (\hat{x}) \right] &\leq \left(\Phi(x_{0}) - \Phi(x_K) \right) + K \eta D \frac{C}{\sqrt{b}} + K \frac{1}{2} M_{\Phi} \eta^2 \; .
\end{align*}
Now, let $C_{x_0}>0$ be an initialization-dependent constant, such that $C_{x_0} > \Phi(x_0) - \E \left[	\Phi(x^\star)\right]$, where $x^{\star}$ is a first-order stationary point.
From $\eta=\frac{1}{\sqrt{K}}$ and $b=K$ we see that
\begin{align*}
\mathbb{E}\left[\mathcal{G}  (\hat{x}) \right] &\leq \frac{1}{\sqrt{K}} \left(C_{x_0} +  D C + \frac{1}{2} M_{\Phi}\right) \; ,
\end{align*}
which shows the desired sublinear convergence rate.
\end{proof}
\begin{cor}
\srfw obtains an $\epsilon$-accurate solution with SFO complexity of $O \left(\frac{1}{\epsilon^4}	\right)$ and RLO complexity of $O \left(\frac{1}{\epsilon^2}	\right)$.
\end{cor}
\begin{proof}
It follows directly from Theorem~\ref{thm:srfw} that \srfw achieves an $\epsilon$-accurate solution after $O \left(\frac{1}{\epsilon^2}	\right)$ iteration, i.e., its RLO complexity is $O \left(\frac{1}{\epsilon^2}	\right)$. For the SFO complexity, note that
\begin{align*}
\sum_{k=0}^{K-1} b_k = Kb = K^2 \lesssim O \left(	\frac{1}{\epsilon^4}\right).
\end{align*}
\end{proof}

\noindent For g-convex objectives, we can obtain a global convergence result in terms of the optimality gap $\Delta_k := \Phi(x_k)-\Phi(x^*)$. Here, \srfw converges at a sublinear rate to the global optimum $\Phi(x^*)$.
\begin{cor}\label{cor:srfw}
If $\Phi$ is g-convex, then under the assumptions of Theorem~\ref{thm:srfw} the optimality gap converges as $\mathbb{E}_{\bm{\xi}} \left[\Delta_{k} \right] = O(1/\sqrt{K})$.
\end{cor}
\begin{proof}
In the proof of Theorem~\ref{thm:srfw}, Eq.~\ref{eq:srfw-1}, note that
\begin{align*}
\Phi(x_{k+1}) &\leq \Phi(x_k) + \eta_k \ip{g_k(x_k)}{\Exp_{x_k}^{-1}(y_k)} + \eta_k \ip{\grad \; \Phi(x_k) - g_k(x_k)}{\Exp_{x_k}^{-1}(y_k)} + \frac{1}{2} M_{\Phi} \eta_k^2 \\
&\overset{(1)}{\leq} \Phi(x_k) + \eta_k \ip{g_k(x_k)}{\Exp_{x_k}^{-1}(x^*)} + \eta_k \ip{\grad \; \Phi(x_k) - g_k(x_k)}{\Exp_{x_k}^{-1}(y_k)} + \frac{1}{2} M_{\Phi} \eta_k^2 
\end{align*}
where (1) follows from $y_k$ being the $\argmin$ as defined in Algorithm~\ref{alg.srfw}. 
Note that in the third term, the Cauchy-Schwartz inequality gives
\begin{align*}
\ip{\grad \Phi(x_k) - g_k (x_k)}{\Exp_{x_k}^{-1}(y_k)} \leq \norm{\grad \Phi(x_k) - g_k (x_k)}
\underbrace{\norm{\Exp_{x_k}^{-1}(y_k)}}_{\leq {\rm diam}(\Xc) =: D} \; .
\end{align*}
Inserting this above and taking expectations, we have
\begin{align*}
\E \left[ \Phi (x_{k+1}) \right] &\leq \E \left[	\Phi(x_k)	\right] + \eta_k \E \left[	\ip{g_k (x_k)}{\Exp_{x_k}^{-1}(x^*)}	\right]+ \eta_k D \underbrace{\E \left[	\norm{\grad \Phi(x_k) - g_k (x_k)}	\right]}_{\leq \frac{C}{\sqrt{b_k}}} + \frac{1}{2} M_{\Phi} \eta_k^2 \\
&\overset{(2)}{\leq} \E \left[	\Phi(x_k)	\right] + \eta_k \E \left[	\ip{ g_k (x_k)}{\Exp_{x_k}^{-1}(x^*)}	\right]+ \eta_k D \frac{C}{\sqrt{b_k}} + \frac{1}{2} M_{\Phi} \eta_k^2 \; ,
\end{align*}
where (2) follows from Lemma~\ref{lem:RM}.  For the second term, we have
\begin{align*}
\E \left[	\ip{ g_k (x_k)}{\Exp_{x_k}^{-1}(x^*)}	\right] = \ip{ \E \left[g_k (x_k)\right]}{\Exp_{x_k}^{-1}(x^*)}\overset{(3)}{ =} \ip{ \grad \Phi (x_k)}{\Exp_{x_k}^{-1}(x^*)} \overset{(4)}{\leq} - \left( \Phi(x_k) - \Phi(x^*) \right) \; ,
\end{align*}
since (3) $g_k (x_k)$ is an unbiased estimate of $\grad \Phi(x_k)$ and (4) the Frank-Wolfe gap upper-bounds the optimality gap, which is a direct consequence of the g-convexity of $\Phi$ (see Eq.~\ref{eq:g-convex}).  Let $\Delta_k := \Phi(x_k) - \Phi(x^*)$ denote the optimality gap. Then,  putting everything together and reording terms, we get
\begin{align*}
\eta_k \E \left[\Delta_k	\right] \leq \E \left[	\Phi(x_k) - \Phi(x_{k+1})\right] +  \eta_k D \frac{C}{\sqrt{b_k}} + \frac{1}{2} M_{\Phi} \eta_k^2 \; .
\end{align*}
Summing, telescoping and inserting the definition of the output ($\hat{x}$ with optimality gap $\Delta_{\hat{k}} = \Phi(\hat{x}) - \Phi (x^*)$), we have
\begin{align*}
\E \left[	\Delta_{\hat{k}}	\right] \left( \sum_k \eta_k	\right) \leq \left( \Phi (x_0) - \Phi (x_K) \right) + DC \sum_{k} \frac{\eta_k}{\sqrt{b_k}} + \frac{1}{2} \sum_k \eta_k^2 \; .
\end{align*}
With the parameter choice $\eta_k = \eta = \frac{1}{\sqrt{K}}$ and $b_k = b = K$, %we have $\sum_k \eta_k = \sqrt{K}$, $\sum_k \eta_k^2 = 1$ and $\sum_k \frac{\eta_k}{\sqrt{b_k}}$, 
the claim follows as
\begin{align*}
\E \left[ \Delta_{\hat{k}} \right] \leq \frac{1}{\sqrt{K}} \left(	\Delta_{x_0} + DC + \frac{1}{2} M_{\Phi}	\right) \; ,
\end{align*}
where $\Delta_{x_0}$ denotes the initial optimality gap, which is a constant whose value depends on the initialization only.
\end{proof}

\noindent A shortcoming of \srfw is its large batch sizes. 
We expect that choosing a non-constant, decreasing step size will reduce the required batch size. 

\subsection{Stochastic variance-reduced Frank-Wolfe}
In addition to the purely stochastic \srfw method we can obtain a stochastic \fw algorithm via a (semi-stochastic) \emph{variance-reduced} approach for problems with a \emph{finite-sum structure}~\eqref{eq:opt2}.  Recall, that in problem~\eqref{eq:opt2}, we assume that the cost function $\Phi$ can be represented as a finite sum $\Phi (x) = \frac{1}{m} \sum_{i=1}^m \phi_i (x)$, where the $\phi_i$ are $L$-smooth (but may be nonconvex). 
We will see that by exploiting the finite-sum structure, we can obtain provably faster FW algorithms. 

%\subsubsection{\svrfw}
%
\begin{algorithm*}[t]
\begin{small}
  \caption{Semi-stochastic variance-reduced Riemannian Frank-Wolfe (\svrfw)}
  \label{alg.svrfw}
  \begin{algorithmic}[1] 
    \State Initialize $\tilde{x}^0 \in \Xc$; assume access to the geodesic map $\gamma: [0,1] \rightarrow \mathcal{M}$. 
    \State Choose number of iterations $S$ and size of epochs $K$ and set minibatch sizes $\lbrace b_k \rbrace_{k=0}^{K-1}$. 
     \For {$s=0,\dots S-1$}
     	%\State $\tilde{x}^s = x_K^s$.
     	\State Compute gradient at $\tilde{x}^s$: $\grad \Phi(\tilde{x}^s) = \frac{1}{N} \sum_{i=1}^m \grad \phi_i (\tilde{x}^s)$.
     	\For {$k=1,\dots K$}
        		\State Sample i.i.d.  $I_k := \left( i_{1} , ..., i_{b_k}\right)\subseteq [m]$ (minibatches).
        		\State $z_{k+1}^{s+1} \gets \argmin_{z \in \Xc} \ip{\frac{1}{b_k} \sum_{j=i_1, ..., i_{b_k}} \grad \phi_j (x_k^{s+1}) - \Gamma_{\tilde{x}^s}^{x_k^{s+1}} \left(	 \grad \phi_j (\tilde{x}^s))	- \grad \Phi (\tilde{x}^s)\right)}{\Exp_{\tilde{x}^s}^{-1}(z)}$
        		\State Compute step size $\eta_k$ and set $x_{k+1}^{s+1} \gets \gamma(\eta_k)$, where $\gamma(0)=x_{k}^{s+1}$ and $\gamma(1)=z_{k+1}^{s+1}$.
        \EndFor
        \State $\tilde{x}^{s+1} = x_K^s$.
     \EndFor
    \State Output $\hat{x}=\tilde{x}_K^S$.
   \end{algorithmic}
   \end{small}
\end{algorithm*}

We first propose \svrfw (Algorithm~\ref{alg.svrfw}), which combines \rfw with a classic variance-reduced estimate of the gradient. This resulting algorithm computes the full gradient at the beginning of each epoch and uses batch estimates within epochs. The variance-reduced gradient estimate guarantees the following bound on the approximation error:\\
\begin{lem}[Goodness of variance-reduced gradient estimate]\label{lem:VR}
Consider the $k^{\text{th}}$ iteration in the $s^{\text{th}}$ epoch and the \emph{stochastic variance-reduced gradient estimate} with respect to a minibatch $I_k = \left(	i_1, \dots,i_{b_k}	\right)$
\begin{align*}
g_k (x_k^{s+1}) = \frac{1}{b_k} \sum_{j=i_1, \dots, i_{b_k}} \grad \phi_j (x_k^{s+1}) - \Gamma_{\tilde{x}^s}^{x_k^{s+1}} \left( \grad \phi_j (\tilde{x}^s) - \grad \Phi(\tilde{x}^s)	\right) \; ,
\end{align*}
with the $\lbrace \phi_i \rbrace$ assumed to be $L$-Lipschitz. Then the expected deviation of the estimate $g_k$ from the true gradient $\grad \Phi$ is bounded as
\begin{align*}
\mathbb{E}_{I_k} \left[	\norm{\grad \Phi(x_k ^{s+1})- g_k (x_k^{s+1})} 	\right] \leq \frac{L}{\sqrt{b_k}} d(x_k^{s+1},\tilde{x}^s) \; .
\end{align*}
\end{lem}
\noindent We again drop the subscript $I_k$, whenever it is clear from context.\\
\begin{proof}
Following Algorithm~\ref{alg.svrfw}, let $I_k = \left( i_1 , \dots, i_{b_k} \right)$ denote the sample in the $k$th iteration of the $s$th epoch. We introduce the shorthands
\begin{align*}
\zeta_k^{s+1} &= \frac{1}{b_k} \sum_{l=1}^{b_k} \grad \phi_{i_l}(x_k^{s+1}) - \Gamma_{\tilde{x}^s}^{x_k^{s+1}} \grad \phi_{i_l} (\tilde{x}^s) \\
\zeta_{k,i_l}^{s+1} &= \grad \phi_{i_l} (x_k^{s+1}) - \Gamma_{\tilde{x}^s}^{x_k^{s+1}} \grad \phi_{i_l} (\tilde{x}^s)	\; ,
\end{align*}
i.e.,  $\zeta_k^{s+1} = \frac{1}{b_k} \sum_{l=1}^{b_k} \zeta_{k,i_l}^{s+1}$. Then we have
\begin{align*}
\mathbb{E} \left[	\norm{\grad \Phi(x_k^{s+1}) - g_k(x_k^{s+1})}^2	\right] &= \mathbb{E} \left[	\norm{	\zeta_k^{s+1} - \grad \Phi(x_k^{s+1}) + \Gamma_{\tilde{x}^s}^{x_k^{s+1}} \grad \Phi(\tilde{x}^s)	}^2	\right] \\
&\overset{(1)}{=} \mathbb{E} \left[	\norm{	\zeta_k^{s+1} - \mathbb{E} \left(	\zeta_k^{s+1}	\right)}^2 \right] \; .
\end{align*}
Here, (1) follows from the following argument:
\begin{align*}
\grad \Phi(x_k^{s+1}) - \Gamma_{\tilde{x}^s}^{x_k^{s+1}} \grad \Phi(\tilde{x}^s)
&\overset{(*)}{=}  \mathbb{E} \left[\frac{1}{b_k} \sum_l \grad \phi_{i_l}(x_k^{s+1}) - \Gamma_{\tilde{x}^s}^{x_k^{s+1}} \grad \phi_{i_l}(\tilde{x}^s) \right] \\
&= \mathbb{E} \left[ \frac{1}{b_k} \sum_l \zeta_{k,i_l}^{s+1}  \right] \\
&= \mathbb{E} \left( \zeta_k^{s+1} \right) \; ,
\end{align*}
where in (*) we used the assumption that the variance-reduced gradient is an unbiased estimate of the full Riemannian gradient.
We further have
%\begin{align*}
%\mathbb{E} \left[	\norm{ \zeta_k^{s+1} - \mathbb{E} \left[ \zeta_k^{s+1}	\right]}^2	\right] 
%&= \mathbb{E} \left[	\norm{\zeta_k^{s+1}}^2	\right] - \underbrace{\norm{\mathbb{E} \left[	\zeta_k^{s+1}	\right]}^2}_{\geq 0} 
%\leq \mathbb{E} \left[	\norm{\zeta_k^{s+1}}^2	\right] \\
%&= \mathbb{E} \left[	\norm{ \frac{1}{b_k} \sum_l \zeta_{k, i_l}^{s+1}}^2	\right]
%\overset{(2)}{\leq} \frac{1}{b_k^2} \mathbb{E}  \left[\sum_l	\norm{   \zeta_{k, i_l}^{s+1}}^2	\right] \\
%&= \frac{1}{b_k^2} \mathbb{E} \left[ \sum_l \underbrace{\norm{ \grad \; \phi_{i_l}(x_k^{s+1}) - \Gamma_{\tilde{x}_s}^{x_k^{s+1}} \grad \; \phi_{i_l} (\tilde{x}^s)}^2}_{\leq L d(x_k^{s+1},\tilde{x}^s)} \right] \\
%&\overset{(3)}{\leq} \frac{b_k L^2 d^2(x_k^{s+1},\tilde{x}^s)}{b_k^2} \; ,
%\end{align*}
\begin{align*}
\mathbb{E} \left[	\norm{ \zeta_k^{s+1} - \mathbb{E} \left[ \zeta_k^{s+1}	\right]}^2	\right] 
&= \mathbb{E} \left[	\norm{\frac{1}{b_k} \sum_l \zeta_{k, i_l}^{s+1} - \mathbb{E} \left[ \frac{1}{b_k}  \sum_l \zeta_{k, i_l}^{s+1}
\right] }^2 \right] \\
&= \frac{1}{b_k^2} \mathbb{E} \left[ \norm{ \sum_l \left(	\zeta_{k, i_l}^{s+1} - \mathbb{E}  \left[  \zeta_{k, i_l}^{s+1} \right]	\right)	}^2 \right] \\
&\overset{(2)}{\leq} \frac{1}{b_k^2} \mathbb{E}  \left[\sum_l	\norm{   \zeta_{k, i_l}^{s+1} - \mathbb{E} \left[	 \zeta_{k, i_l}^{s+1} \right]}^2	\right] \\
&\leq \frac{1}{b_k^2} \mathbb{E}  \left[\sum_l	\norm{   \zeta_{k, i_l}^{s+1} }^2	\right] \\
&= \frac{1}{b_k^2} \mathbb{E} \left[ \sum_l \underbrace{\norm{ \grad \; \phi_{i_l}(x_k^{s+1}) - \Gamma_{\tilde{x}_s}^{x_k^{s+1}} \grad \; \phi_{i_l} (\tilde{x}^s)}^2}_{\leq L d(x_k^{s+1},\tilde{x}^s)} \right] \\
&\overset{(3)}{\leq} \frac{b_k L^2 d^2(x_k^{s+1},\tilde{x}^s)}{b_k^2} \; .
\end{align*}
For (2), recall that $ \zeta_{k, i_l}^{s+1} = \grad \phi_i (x_k) - \Gamma_{\tilde{x}^s}^{x_k^{s+1}} \grad \phi_i (\tilde{x}^s)$ and therefore $\mathbb{E}\left[  \zeta_{k, i_l}^{s+1} - \mathbb{E}\left[  \zeta_{k, i_l}^{s+1} \right]	\right] = 0$ for all $i_l \in I_k$. The inequality follows then from Remark~\ref{eq:xi-sum-squares}.  Inequality (3) follows from the assumption that the $\phi_i$ are $L$-Lipschitz smooth.  
This shows
\begin{align*}
\mathbb{E} \left[\norm{ \grad \Phi(x_k^{s+1}) - g_k (x_k^{s+1})}^2	\right] \leq \frac{L^2}{b_k} d^2(x_k^{s+1}, \tilde{x}^s) \; .
\end{align*}
Jensen's inequality gives
\begin{align*}
\mathbb{E} \left[\norm{ \grad \Phi(x_k^{s+1}) -  g_k (x_k^{s+1})}^2	\right]  \geq \mathbb{E} \left[\norm{ \grad \Phi(x_k^{s+1}) -  g_k (x_k^{s+1})}	\right]^2 \; ,
\end{align*}
and, putting everything together and taking the square root on both sides, the claim follows as
\begin{align*}
\mathbb{E} \left[\norm{ \grad \Phi(x_k^{s+1}) - g_k (x_k^{s+1})}	\right] \leq \frac{L}{\sqrt{b_k}} d(x_k^{s+1},\tilde{x}^s) \; .
\end{align*}
\end{proof}

\noindent Using Lemma~\ref{lem:VR} we can recover the following sublinear convergence rate:\\
\begin{theorem}\label{thm:conv-svrfw}
With steps size $\eta_k = \frac{1}{\sqrt{KS}}$ and constant batch sizes $b_k=K^2$, Algorithm~\ref{alg.svrfw} converges in expectation with $\mathbb{E}_{I_k} \left[\mathcal{G}(\hat{x})\right] =O \left(\frac{1}{\sqrt{KS}}\right)$. Here, $\mathcal{G}(x)$ again denotes the Frank-Wolfe gap as defined in Eq.~\ref{def:FW-gap}.\\
\end{theorem}
\begin{proof}(Theorem~\ref{thm:conv-svrfw})
Let again
\begin{align}
g_k (x_k^{s+1}) = \frac{1}{b_k} \sum_j \grad \phi_j (x_k^{s+1}) - \Gamma_{\tilde{x}^s}^{x_k^{s+1}} \left( \grad \phi_j (\tilde{x}^s) - \grad \Phi(\tilde{x}^s)			\right)
\end{align}
denote the variance-reduced gradient estimate in the $k^{th}$ iteration  of the $s^{th}$ epoch. Then
\begin{align}\label{eq:svrfw-1}
\Phi(x_{k+1}^{s+1}) &\overset{(1)}{\leq} \Phi(x_k^{s+1}) + \eta_k \ip{\grad \; \Phi(x_k^{s+1})}{\Exp_{x_k^{s+1}}^{-1}(y_k)} + \frac{1}{2} M_{\Phi} \eta_k^2 \\
&\overset{(2)}{\leq} \Phi(x_k^{s+1}) + \eta_k \ip{g_k(x_k^{s+1})}{\Exp_{x_k^{s+1}}^{-1}(y_k)} \\
& \qquad+ \eta_k \ip{\grad \; \Phi(x_k^{s+1}) - g_k(x_k^{s+1})}{\Exp_{x_k^{s+1}}^{-1}(y_k)} + \frac{1}{2} M_{\Phi} \eta_k^2 \; 
\end{align}
Here, (1) follows from Lemma~\ref{A.smooth} and (2) from ``adding a zero" with respect to $g_k$. 
We then apply Cauchy-Schwartz to the inner product and make use of the fact that the geodesic distance between points in $\Xc$ is bounded by its diameter:
\begin{align}
\ip{\grad \; \Phi(x_k^{s+1}) - g_k(x_k^{s+1})}{\Exp_{x_k^{s+1}}^{-1}(y_k)} \leq \norm{\grad \; \Phi(x_k^{s+1}) - g_k(x_k^{s+1})} \cdot \underbrace{\norm{\Exp_{x_k^{s+1}}^{-1}(y_k)}}_{\leq {\rm diam}(\Xc)} \; .
\end{align}
This gives (with $D := {\rm diam}(\Xc)$)
\begin{align*}
\Phi(x_{k+1}^{s+1}) &\leq \Phi(x_k^{s+1}) + \eta_k \ip{g_k(x_k^{s+1})}{\Exp_{x_k^{s+1}}^{-1}(y_k)} + \eta_k D \norm{\grad \; \Phi(x_k^{s+1}) - g_k(x_k^{s+1})} + \frac{1}{2} M_{\Phi} \eta_k^2 \; .
\end{align*}
Taking expectations, we have
\begin{align}\label{eq:3-3-exp}
\E \left[\Phi(x_{k+1}^{s+1}) \right] &\leq \E \left[ \Phi(x_k^{s+1}) \right] + \eta_k \E \left[ \ip{g_k(x_k^{s+1})}{\Exp_{x_k^{s+1}}^{-1}(y_k)} \right]\\ &\quad+ \eta_k D \E \left[ \norm{\grad \; \Phi(x_k^{s+1}) - g_k(x_k^{s+1})} \right] + \frac{1}{2} M_{\Phi} \eta_k^2 \\
&\overset{(3)}{\leq}  \E \left[ \Phi(x_k^{s+1}) \right]  - \eta_k \E \left[ \hat{\mathcal{G}}(x_k^{s+1})	\right] + \eta_k D \frac{L}{\sqrt{b_k}} \E \left[	d(x_k^{s+1}, \tilde{x}^s)	\right]  + \frac{1}{2} M_{\Phi} \eta_k^2 \; ,
\end{align}
where (3) follows from applying the definition of the stochastic Frank-Wolfe gap to the second term and Lemma~\ref{lem:VR} to the third term.  

For the following analysis,  define for $k = 1, \dots, K$ and a fixed epoch $s \in [S]$
\begin{align}
R_k &:= \E \left[ \Phi (x_k^{s+1})  + c_k d(x_k^{s+1}, \tilde{x}^s)\right] \\
c_k &= c_{k+1} + \eta_k D \frac{L}{\sqrt{b_k}} \qquad (c_K = 0) \; .
\end{align}
With that and inequality~\ref{eq:3-3-exp}, we have
\begin{align*}
R_{k+1} &= \E \left[	\Phi(x_{k+1}^{s+1})	\right] + c_{k+1} \E \left[	 d(x_{k+1}^{s+1}, \tilde{x}^s)	\right] \\
&\leq \E \left[ \Phi(x_k^{s+1})	\right] - \eta_k \E \left[ \hat{\mathcal{G}}(x_k^{s+1})	\right] + \eta_k D \frac{L}{\sqrt{b_k}} \E \left[	d(x_k^{s+1}, \tilde{x}^s)	\right] + \frac{1}{2} M_{\Phi} \eta_k^2 \\
&\quad+ c_{k+1} \underbrace{\E \left[	d(x_{k+1}^{s+1}, \tilde{x}^s)	\right]}_{\overset{(4)}{\leq} \E \left[	d(x_{k+1}^{s+1}, x_k^{s+1})	+ d(x_k^{s+1}, \tilde{x}^s) \right] } \\
&\leq \underbrace{\left( \E \left[	\Phi (x_k^{s+1})\right] + \underbrace{\left( c_{k+1} + \eta_k D \frac{L}{\sqrt{b_k}} \right)}_{=c_k}	\E \left[	d(x_k^{s+1}, \tilde{x}^s)	\right]	\right)}_{= R_k} - \eta_k \E \left[ \hat{\mathcal{G}}(x_k^{s+1})	\right] \\
&\quad+  c_{k+1} \E \left[	d(x_{k+1}^{s+1}, x_k^{s+1})	\right] + \frac{1}{2} M_{\Phi} \eta_k^2 \\
&\overset{(5)}{\leq} R_k - \eta_k \E \left[ \hat{\mathcal{G}}(x_k^{s+1})	\right] +  c_{k+1}\eta_k D + \frac{1}{2} M_{\Phi} \eta_k^2 \; .
\end{align*}
where (4) follows by adding a zero and applying the triangle-inequality and (5) from the definition of $R_k$ and the definition of the update step via the geodesic map $\gamma$ (see Algorithm~\ref{alg.svrfw})
\begin{align}
\E \left[	d(x_{k+1}^{s+1}, x_k^{s+1})	\right] \leq \eta_k \mathbb{E} \left[ \norm{\Exp_{x_k}^{-1}(z_k)}	\right]  \leq  \eta_{k} D \; .
\end{align}\label{eq:diam-eta}
Telescoping within the epoch $s+1$ we get (with $\eta_k = \eta$ and $b_k = b$ for $k=0, \dots, K-1$)
\begin{align*}
R_K &\leq R_0 - \sum_k \eta_k \E \left[ \hat{\mathcal{G}}(x_k^{s+1})	\right] + \frac{1}{2} M_{\Phi} \sum_k \eta_k^2 + D \sum_k \eta_k c_{k+1} \\
&= R_0 - \eta \sum_k \E \left[ \hat{\mathcal{G}}(x_k^{s+1})	\right] + \frac{1}{2} M_{\Phi} \eta^2 K + D \eta \sum_k c_{k+1} \\
&= R_0 - \eta \sum_k \E \left[ \hat{\mathcal{G}}(x_k^{s+1})	\right] + \frac{1}{2} M_{\Phi} \eta^2 K + \frac{\eta^2 D^2 L}{\sqrt{b}} \frac{K(K-1)}{2} \; .
\end{align*}
This gives
\begin{align*}
\E \left[	\Phi(x_K^{s+1})	\right] \leq \E \left[ \Phi(x_K^{s}) \right] - \eta \sum_k \E \left[ \hat{\mathcal{G}}(x_k^{s+1})	\right] + \frac{1}{2} M_{\Phi} \eta^2 K + \frac{\eta^2 D^2 L}{\sqrt{b}} \frac{K(K-1)}{2} \; .
\end{align*}
Finally, telescoping over all epochs $s=0, \dots, S-1$, we get
\begin{align*}
 \E \left[	\Phi(x_K^{S})	\right]
&\leq \E \left[	\Phi(x_0) \right] - \eta \sum_s \sum_k \E \left[	\hat{\mathcal{G}}(x_k^{s+1})	\right] + \frac{1}{2} M_{\Phi} \eta^2 K S + \frac{\eta^2 D LS}{\sqrt{b}} \frac{K(K-1)}{2} \; .
\end{align*}
Reordering terms and using the definition of the output in Algorithm~\ref{alg.svrfw} (and the fact that $\mathbb{E} \left[\mathbb{E} \left[	\hat{\mathcal{G}}(x_K^S)\right] \right]= \mathbb{E} \left[\mathcal{G}(\hat{x}) \right]$), this gives
\begin{align*}
KS \eta \E \left[	\mathcal{G}(\hat{x})\right] &\leq \Phi(x_0) - \E \left[ \Phi(x_K^S) \right] + \frac{1}{2} M_{\Phi} \eta^2 KS + \frac{\eta^2 D LS}{\sqrt{b}} \frac{K(K-1)}{2} \; ,
\end{align*}
from which the claim follows with $\eta = \frac{1}{\sqrt{KS}}$ and $b= K^2$ as
\begin{align*}
\E \left[ \mathcal{G}(\hat{x}) \right] \leq \frac{1}{\sqrt{KS}} \left(
C_{x_0} + \frac{1}{2} (M_{\Phi} + D^2 L) \right) \; ,
\end{align*}
where $C_{x_0}>0$ is an initialization-dependent constant, such that $C_{x_0} > \Phi(x_0) - \E \left[	\Phi(x^\star)\right] > \Phi(x_0) - \E \left[	\Phi(x_K^S)\right]$, where $x^\star$ is a first-order stationary point.
\end{proof}

Choosing a suitable minibatch size is critical to achieving a good performance with variance-reduced approaches, such as \svrfw. In Algorithm~\ref{alg.svrfw} this translates into a careful choice of $K$ with respect to $m$: If $K$ is too small,  the complexity of the algorithm may be dominated by the cost of recomputing the full gradient frequently. If $K$ is too large, than computing the gradient estimates will be expensive too. We propose to set $K = \lceil m^{1/3} \rceil$, following a convention in the Euclidean \fw literature. With that, we get the following complexity guarantees:
\begin{cor}
\svrfw with $K = \lceil m^{1/3} \rceil$ obtains an $\epsilon$-accurate solution with IFO complexity of $O \left(m + \frac{m^{2/3}}{\epsilon^2}	\right)$ and RLO complexity of $O \left(\frac{1}{\epsilon^2}	\right)$.
\end{cor}
\begin{proof}
It follows directly from Theorem~\ref{thm:conv-svrfw} that \svrfw has an LO complexity of $O \left(\frac{1}{\epsilon^2}	\right)$. For the IFO complexity, note that
\begin{align*}
\sum_{s=0}^{S-1} \left( m + \sum_{k=1}^{K-1} b_k \right) = \sum_{s=0}^{S-1} \left( m + K b \right) \lesssim O \left( m + \frac{K^2}{\epsilon^2} \right) = O \left(m + \frac{m^{2/3}}{\epsilon^2} \right) \; ,
\end{align*}
where the last equality follows from setting $K = \lceil m^{1/3} \rceil$.
\end{proof}

\noindent Analogously to \srfw, \svrfw converges sublinearly to the global optimum, if the objective is g-convex. As before, we use $\Delta_k = \Phi(x_k)-\Phi(x^*)$.
\begin{cor}\label{cor:svrfw}
If $\Phi$ is g-convex, then in the setting of Theorem~\ref{thm:conv-svrfw} the optimality gap converges as $\mathbb{E}_{I_k} \left[\Delta_{k} \right] = O(1/\sqrt{KS})$.
\end{cor}
\noindent The proofs are very similar to that of Corollary~\ref{cor:srfw}.\\

A significant shortcoming of the semi-stochastic approach is the need for repeated computation of the full gradient which limits its scalability. In the following section, we introduce an improved version that circumvents these costly computations.

\subsection{Improved gradient estimation with \textsc{Spider}}
\begin{algorithm*}[ht]
\begin{small}
  \caption{\spiderfw}
  \label{alg.spider}
  \begin{algorithmic}[1] 
    \State Initialize $x_0 \in \Xc$, number of iterations $K$, size of epochs $n$. Assume access to $\gamma: [0,1] \rightarrow \mathcal{M}$. 
     \For {$k=0,1,\dots K-1$}
      \If {${\rm mod}(k,n)=0$}
      	\State Sample i.i.d. $S_1 = \lbrace \xi_{1} , ..., \xi_{\vert S_1 \vert}\rbrace$ (for \srfw) or $S_1 = \left( i_1, \dots, i_{\vert S_1 \vert} \right)$ (for \svrfw) with predefined $\vert S_1 \vert$.
      	\State Compute gradient $g_k \gets \grad \Phi_{S_1}(x_k)$.
      \Else
      	\State $\vert S_2 \vert \gets \lceil  \min \lbrace	 m, \frac{2nL^2 \norm{\Exp_{x_{k-1}}^{-1}(x_k)}}{\epsilon^2}	\rbrace  \rceil$
     	\State Sample i.i.d. $S_2 = \lbrace \xi_{1} , ..., \xi_{\vert S_2 \vert}\rbrace$ (for \srfw) or  $S_2 = \left( i_1, \dots, i_{\vert S_2 \vert} \right)$ (for \svrfw).
		\State  Compute gradient $g_k \gets \grad \Phi_{S_2}(x_k) - \Gamma_{x_{k-1}}^{x_k} \left(	\grad \Phi_{S_2}(x_{k-1}) - g_{k-1}	\right)$.
	\EndIf
	\State $z_{k+1} \gets \argmin_{z \in \Xc} \ip{g_k}{\Exp_{x_k}^{-1}(z)}$.
    \State $x_{k+1} \gets \gamma(\eta_k)$, where $\gamma(0)=x_{k}$ and $\gamma(1)=z_{k+1}$.
    \EndFor
    \State Output $\hat{x}$ chosen uniformly at random from $\lbrace x^k	\rbrace_{k=0}^{K-1}$.
   \end{algorithmic}
   \end{small}
\end{algorithm*}
Recently, \citet{sarah} and \citet{spider} introduced \spider (also known as \textsc{Sarah})
as an efficient way of estimating the (Euclidean) gradient in stochastic optimization tasks. Based on the idea of variance-reduction, the algorithm iterates between gradient estimates with different sample size. In particular, it recomputes the gradient at the beginning of each epoch with a larger (constant) batch size; the smaller batch sizes within epochs decrease as we move closer to the optimum.  This technique was studied for Riemannian Gradient Descent in~\citep{rspider1} and~\citep{rspider2}. In the following, we will introduce an improved variance-reduced \textsc{Stochastic} \rfw using \spider. 
Let
\begin{equation}
\grad \Phi_{S}(x) = 
\begin{cases}
\frac{1}{\vert S \vert} \sum_{i=1}^{\vert S \vert} \grad \phi(x,\xi_i), &{\rm stochastic} \\
\frac{1}{\vert S \vert} \sum_{i=1}^{\vert S \vert} \grad \phi_i (x), &{\rm finite-sum}
\end{cases}
\end{equation}
denote the gradient estimate with respect to a sample $S=\lbrace \xi_1, \dots, \xi_{\vert S \vert} \rbrace$ (for stochastic objectives) or $S=\left(i_1, \dots, i_{\vert S \vert}	\right)$ (for objectives with finite sum form).
Furthermore,  we make the following parameter choice ($K$ denoting the number of iterations): 
\begin{align}\label{params:spider}
 \eta&=\frac{1}{\sqrt{K}} \; \; {\rm (step \; size)} \\ 
 n&=\sqrt{K}=\frac{1}{\epsilon} \; \; (\# \; {\rm epochs}) \\
  \vert S_1 \vert &=
  \begin{cases}
   \frac{2C^2}{\epsilon^2}, &{\rm stochastic} \\
   \frac{2L^2D^2}{\epsilon^2}, &{\rm finite-sum}
  \end{cases}
\end{align}
Here, $\epsilon$ characterizes the goodness of the gradient estimate. 
$\vert S_2 \vert$ is recomputed in each iteration as given in Algorithm~\ref{alg.spider}. Note that here $m$ is determined by the number of terms in the finite-sum approximation or we set $m=\infty$ in the  stochastic case. 

We start by analyzing the goodness of the \spider gradient estimate $g_k$,  which is central to our convergence analysis. For ${\rm mod}(k,n) = 0$ an upper bound is given by Lemmas~\ref{lem:RM} and~\ref{lem:VR}.  The critical part is to analyze the case ${\rm mod}(k,n) \neq 0$. Let $\mathcal{F}_k$ be the sigma-field generated by the $x_k$. First, we show that the differences $\left(g_k - \grad \Phi(x_k)\right)_k$ form a martingale with respect to $(\mathcal{F}_k)_k$ (Lemma~\ref{eq:approx-martingale}). Then, using a classical property of $L^2$-martingales (Remark~\ref{rmk:martingale-prop}), we can prove the following bound on the approximation error:
\begin{lem}[Goodness of \spider-approximation]\normalfont
\label{lem.spider}
The expected deviation of the estimate $g_k$ from the true gradient $\grad \Phi$ as defined in Algorithm~\ref{alg.spider} (${\rm mod}(k,n) \neq 0$) is bounded as $\mathbb{E} \left[ \norm{g_k	- \grad \Phi(x_k)} \vert \mathcal{F}_k \right] \leq \epsilon$.
\end{lem}
\noindent We first show that the differences form a martingale:
\begin{lem}\label{eq:approx-martingale}
The differences of the gradient estimates $g_k$ from the true gradients $\grad \Phi$, i.e., \\ $\left( g_k - \grad \Phi(x_k)	\right)_k$, form a martingale with respect to the filtration $\left( \mathcal{F}_k \right)_k$. 
\end{lem}
\begin{proof}
\begin{align*}
\mathbb{E}\left[ g_k - \grad \Phi(x_k) \vert \mathcal{F}_k \right] 
&= \mathbb{E} \left[ \grad \Phi_{S_2} (x_k) - \Gamma_{x_{k-1}}^{x_k} \left(\grad \Phi_{S_2}(x_{k-1})  - g_{k-1}\right) - \grad \Phi(x_k) \vert \mathcal{F}_k \right] \\
&= \underbrace{\mathbb{E} \left[ \grad \Phi_{S_2}(x_k) - \grad \Phi(x_k) \vert	\mathcal{F}_k	\right] }_{=0}
+ \mathbb{E}\left[	 \Gamma_{x_{k-1}}^{x_k} g_{k-1} - \grad \Phi_{S_2}(x_{k-1})	 \vert \mathcal{F}_k \right] \\
&\overset{*}{=} \Gamma_{x_{k-1}}^{x_k} g_{k-1} - \grad \Phi (x_{k-1}) \; ,
\end{align*}
where (*) follows from $\mathbb{E} \left[ \grad \Phi_{S_2}(x_k)	\vert \mathcal{F}_k \right] = \grad \Phi (x_k)$, since $ \grad \Phi_{S_2}(x_k)$ is assumed to be an unbiased estimate.
\end{proof}
\begin{rmk}\label{rmk:martingale-prop}\normalfont
Let $M = (M_k)_k$ denote an $L^2$-martingale.  The orthogonality of increments, i.e.,
\begin{align*}
\ip{M_t - M_s}{M_v - M_u} = 0 \qquad (v \geq u \geq t \geq s) \; .
\end{align*}
implies that
\begin{align*}
\mathbb{E} \left[	M_k^2	\right] = \mathbb{E} \left[	M_{k-1}^2	\right] +  \mathbb{E} \left[	\left(M_{k} - M_{k-1} \right)^2 	\right] \; .
\end{align*}
Therefore, we have recursively
\begin{align*}
\mathbb{E} \left[	M_k^2	\right] = \mathbb{E} \left[	M_0^2	\right] + \sum_{i=1}^k \mathbb{E} \left[	\left(M_{i} - M_{i-1} \right)^2 	\right] \; .
\end{align*}
\end{rmk}

\noindent We can now prove Lemma~\ref{lem.spider}:\\
\begin{proof} (Lemma~\ref{lem.spider})
We consider two cases:
\begin{enumerate}
\item \mybox{${\rm mod}(m,k)=0$} 
For stochastic objectives, we have
\begin{align}\label{eq:lem.spider.1}
\mathbb{E} \left[ \norm{g_k - \grad \Phi(x_k)}^2	\vert \mathcal{F}_k \right] &= \mathbb{E} \left[	\norm{\grad \Phi_{S_1}(x_k) - \grad \Phi(x_k)}^2 \vert \mathcal{F}_k	\right] 
\overset{(1)}{\leq}  \frac{C^2}{\vert S_1 \vert} = \frac{C^2 \epsilon^2}{2C^2} = \frac{\epsilon^2}{2} \; ,
\end{align}
where (1) follows from Lemma~\ref{lem:RM}.  For objectives with finite-sum form, we have
\begin{align}\label{eq:lem.spider.1-fs}
\mathbb{E} \left[ \norm{\grad \Phi_{S_1}(x_k) - \grad \Phi(x_k)	}^2 \vert \mathcal{F}_k \right] &\overset{(2)}{\leq} \frac{L^2 D^2}{\vert S_1 \vert} = \frac{L^2 D^2 \epsilon^2}{2 L^2 D^2} = \frac{\epsilon^2}{2} \; ,
\end{align}
where (2) follows from Lemma~\ref{lem:VR}.  
\item \mybox{${\rm mod}(m,k) \neq 0$} 
We have
\begin{align*}
&\mathbb{E} \left[\norm{g_k - \grad \Phi(x_k)}^2 \vert \mathcal{F}_k \right]  \\
&\overset{(3)}{=} \mathbb{E} \left[ \norm{ \Gamma_{x_{k-1}}^{x_k} \left(g_{k-1} - \grad \Phi(x_{k-1}) \right)}^2 \vert \mathcal{F}_k \right]\\
&\quad+ \mathbb{E} \left[\norm{g_k - \grad \Phi(x_k) - \Gamma_{x_{k-1}}^{x_k} \left(g_{k-1}	- \grad \Phi (x_{k-1}) \right) }^2 \vert \mathcal{F}_k \right] \\
&\overset{(4)}{=} \mathbb{E} \left[  \norm{\Gamma_{x_{k-1}}^{x_k} \left( g_{k-1} - \grad \Phi(x_{k-1}) \right)}^2 \vert \mathcal{F}_k \right] \\
&\quad + \mathbb{E} \left[\vert \! \vert \grad \Phi_{S_2}(x_k) \right. \\
&\qquad \left. - \Gamma_{x_{k-1}}^{x_k} \left(	\grad \Phi_{S_2}(x_{k-1})  - g_{k-1}		\right)
- \grad \Phi(x_k) - \Gamma_{x_{k-1}}^{x_k} \left(  g_{k-1}	- \grad \Phi (x_{k-1})\right) \vert \! \vert^2 \vert \mathcal{F}_k \right]  \\
&= \mathbb{E} \left[  \norm{\Gamma_{x_{k-1}}^{x_k} \left( g_{k-1} - \grad \Phi(x_{k-1}) \right)}^2 \vert \mathcal{F}_k \right] \\
&\quad + \mathbb{E} \left[\norm{\grad \Phi_{S_2}(x_k) - \grad \Phi(x_k)
- \Gamma_{x_{k-1}}^{x_k} \left(	\grad \Phi_{S_2}(x_{k-1})  - \grad \Phi (x_{k-1})\right) }^2 \vert \mathcal{F}_k \right]  \; .
\end{align*}
In the chain of inequalities, (3) follows from Remark~\ref{rmk:martingale-prop} and (4) from substituting $g_k$ according to Algorithm~\ref{alg.spider}. 

In the following, we assume that $\Phi$ is a \emph{stochastic function}. Analogous arguments hold, if $\Phi$ has a finite-sum structure. We introduce the shorthand
\begin{align*}
\zeta_i = 
\grad \phi(x_k,\xi_i) - \grad \Phi(x_k) - \Gamma_{x_{k-1}}^{x_k} \left(	\grad \phi(x_{k-1},\xi_i) - \grad \Phi(x_{k-1}) 		\right) \; .
\end{align*}
Then, we get for the second term
\begin{align*}
&\mathbb{E} \left[\norm{\grad \Phi_{S_2} (x_k) - \grad \Phi(x_k) - \Gamma_{x_{k-1}}^{x_k} \left( \grad \Phi_{S_2} (x_{k-1}) - \grad \Phi(x_{k-1})	\right) }^2 \vert \mathcal{F}_k \right] \\
&= \mathbb{E} \left[	\norm{\frac{1}{\vert S_2 \vert} \sum_{i=1}^{\vert S_2 \vert} \zeta_i	}^2 \vert \mathcal{F}_k	\right] 
= \frac{1}{\vert S_2 \vert^2}  \mathbb{E} \left[	\norm{\sum_{i=1}^{\vert S_2 \vert}	\zeta_i	}^2 \vert \mathcal{F}_k	\right] \\
&\overset{(5)}{\leq} \frac{1}{\vert S_2 \vert^2} \mathbb{E} \left[	\left( \sum_{i=1}^{\vert S_2 \vert} \norm{	\zeta_i	} \right)^2 \vert \mathcal{F}_k	\right] 
\overset{(6)}{=} \frac{1}{\vert S_2 \vert^2}  \mathbb{E} \left[	\sum_{i=1}^{\vert S_2 \vert} \norm{	\zeta_i	}^2 \vert \mathcal{F}_k	\right] \\
&= \frac{1}{\vert S_2 \vert^2} 	\sum_{i=1}^{\vert S_2 \vert} \mathbb{E} \left[ \norm{	\zeta_i	}^2 \vert \mathcal{F}_k	\right] 
\overset{(7)}{=} \frac{1}{\vert S_2 \vert} \mathbb{E} \left[ \norm{	\zeta_i	}^2 \vert \mathcal{F}_k	\right] \\
&= \frac{1}{\vert S_2 \vert} \mathbb{E} \left[  \norm{\grad \phi(x_k,\xi) - \grad \Phi(x_k) - \Gamma_{x_{k-1}}^{x_k} \left(	\grad \phi(x_{k-1},\xi \right) - \grad \Phi(x_{k-1})}^2 	\vert \mathcal{F}_k	\right] \; .
\end{align*}
where (5) follows from the triangle-inequality, (6) from $\mathbb{E}\left[	\zeta_i	\right]=0$,  see Equation~\ref{eq:xi-sum-squares}; and (7) from the $\zeta_i$ being i.i.d.  Note, that
\begin{align*}
 \mathbb{E} \left[  \grad \phi(x_k,\xi) \vert \mathcal{F}_k \right] &= \grad \Phi(x_k) \\
 \mathbb{E} \left[ \Gamma_{x_{k-1}}^{x_k} \grad \phi(x_{k-1},\xi )  \vert \mathcal{F}_k \right] &= \Gamma_{x_{k-1}}^{x_k} \grad \Phi(x_{k-1}) \; .
\end{align*}
With this, we have
\begin{align*}
&\mathbb{E} \left[  \norm{\grad \phi(x_k,\xi) - \grad \Phi(x_k) - \Gamma_{x_{k-1}}^{x_k} \left(	\grad \phi(x_{k-1},\xi \right) - \grad \Phi(x_{k-1}) }^2	\vert \mathcal{F}_k	\right] \\
&= \mathbb{E} \left[  \norm{\grad \phi(x_k,\xi) - \Gamma_{x_{k-1}}^{x_k} 	\grad \phi(x_{k-1},\xi)  - \underbrace{\left( \grad \Phi(x_k) - \Gamma_{x_{k-1}}^{x_k}  \grad \Phi(x_{k-1}) \right)}_{= \mathbb{E} \left[	\grad \phi(x_k, \xi) - \Gamma_{x_{k-1}}^{x_k} \grad \phi (x_k, \xi) \vert \mathcal{F}_k \right]}   }^2	\vert \mathcal{F}_k	\right] \\
&= \mathbb{E} \left[  \norm{\grad \phi(x_k,\xi) - \Gamma_{x_{k-1}}^{x_k} 	\grad \phi(x_{k-1},\xi)}^2 \right] - \underbrace{\norm{\mathbb{E} \left[	\grad \phi(x_k, \xi) - \Gamma_{x_{k-1}}^{x_k} \grad \phi (x_k, \xi) \vert \mathcal{F}_k	\right]}^2}_{\geq 0} \\
&\leq \mathbb{E} \left[  \norm{\grad \phi(x_k,\xi) - \Gamma_{x_{k-1}}^{x_k} 	\grad \phi(x_{k-1},\xi)}^2  \vert \mathcal{F}_k	\right] \; .
\end{align*}
In summary, we have for the second term 
\begin{align}
&\mathbb{E} \left[\norm{\grad \Phi_{S_2} (x_k) - \grad \Phi(x_k) - \Gamma_{x_{k-1}}^{x_k} \left( \grad \Phi_{S_2} (x_{k-1}) - \grad \Phi(x_{k-1})	\right) }^2 \vert \mathcal{F}_k \right] \\
&\leq \mathbb{E} \left[  \norm{\grad \phi(x_k,\xi) - \Gamma_{x_{k-1}}^{x_k} 	\grad \phi(x_{k-1},\xi)}^2  \vert \mathcal{F}_k	\right] \; .
\end{align}
Putting everything together, we get
\begin{align*}
&\mathbb{E} \left[\norm{g_k - \grad \Phi(x_k)}^2 \vert \mathcal{F}_k \right] \\
&\leq \mathbb{E} \left[  \norm{\Gamma_{x_{k-1}}^{x_k} \left( g_{k-1} - \grad \Phi(x_{k-1}) ]\right)}^2 \vert \mathcal{F}_k \right] 
+\frac{1}{\vert S_2 \vert} \mathbb{E} \left[  \norm{\grad \phi(x_k,\xi)  - \Gamma_{x_{k-1}}^{x_k} \left(	\grad \phi(x_{k-1},\xi \right)}^2 	\vert \mathcal{F}_k	\right]\\
&\overset{(8)}{\leq} \mathbb{E} \left[  \norm{\Gamma_{x_{k-1}}^{x_k} \left( g_{k-1} - \grad \Phi(x_{k-1})\right) }^2 \vert \mathcal{F}_k \right] +\frac{1}{\vert S_2 \vert} L^2  \norm{\Exp_{x_{k-1}}^{-1}(x_k)}	\vert \mathcal{F}_k \\
&\overset{(9)}{=} \mathbb{E} \left[  \norm{\Gamma_{x_{k-1}}^{x_k} \left(g_{k-1} - \grad \Phi(x_{k-1})\right)}^2 \vert \mathcal{F}_k \right] 
+\frac{\epsilon^2}{2m L^2 \norm{\Exp_{x_{k-1}}^{-1}(x_k)}} L^2  \norm{\Exp_{x_{k-1}}^{-1}(x_k)} \\
&= \mathbb{E} \left[  \norm{\Gamma_{x_{k-1}}^{x_k} \left(g_{k-1} - \grad \Phi(x_{k-1})\right)}^2 \vert \mathcal{F}_k \right] 
+ \frac{\epsilon^2}{2m}
\end{align*}
where (8) follows from $\phi$ being $L$-Lipschitz and (9) from the choice of $\vert S_2 \vert$ in Algorithm~\ref{alg.spider}. Recursively going back to the beginning of the epoch (see Remark~\ref{rmk:martingale-prop}), we get (with $k_0=\lfloor \frac{k}{m} \rfloor m$):
\begin{align*}
\mathbb{E} \left[\norm{g_k - \grad \Phi(x_k)}^2 \vert \mathcal{F}_k \right] \leq \underbrace{\mathbb{E} \left[\norm{g_{k_0} - \grad \Phi(x_{k_0})}^2 \vert \mathcal{F}_{k_0} \right]}_{\leq \frac{\epsilon^2}{2} \; {\rm Eq.(~\ref{eq:lem.spider.1})}} + m \frac{\epsilon^2}{2m} \leq \epsilon^2.  
\end{align*}
With Jensen's inequality, we have
\begin{align*}
\left(\mathbb{E} \left[\norm{g_k - \grad \Phi(x_k)} \vert \mathcal{F}_k \right] \right)^2 \leq \mathbb{E} \left[\norm{g_k - \grad \Phi(x_k)}^2 \vert \mathcal{F}_k \right]
\leq \epsilon^2 \; ,
\end{align*}
which gives
\begin{align*}
\mathbb{E} \left[\norm{g_k - \grad \Phi(x_k)} \vert \mathcal{F}_k \right] \leq \epsilon \; .
\end{align*}
Analogously, if $\Phi$ has a finite-sume structure with component functions $\phi_i$ that are $L$-Lipschitz,  we get
\begin{align*}
\mathbb{E} \left[\norm{g_k - \grad \Phi(x_k)}^2 \vert \mathcal{F}_k \right] \leq \underbrace{\mathbb{E} \left[\norm{g_{k_0} - \grad \Phi(x_{k_0})}^2 \vert \mathcal{F}_{k_0} \right]}_{\leq \frac{\epsilon^2}{2} \; {\rm Eq.(~\ref{eq:lem.spider.1-fs})}} + m \frac{\epsilon^2}{2m} \leq \epsilon^2 \; ,
\end{align*}
from which, again, with Jensen's inequality the claim follows as
\begin{align*}
\mathbb{E} \left[\norm{g_k - \grad \Phi(x_k)} \vert \mathcal{F}_k \right] \leq \epsilon \; .
\end{align*}
\end{enumerate}
\end{proof}

\noindent With this preparatory work, we arrive at the main result for this section: We show that \spiderfw attains a global sublinear convergence rate
for nonconvex objectives.\\
\begin{theorem}[Convergence \spiderfw]\label{conv:spider}
  With the parameter choices (~\ref{params:spider}), Algorithm~\ref{alg.spider} converges in expectation with rate $\mathbb{E}\left[	\mathcal{G} (\hat{x})	\right] = O \left(	\frac{1}{\sqrt{K}}	\right)$.\\
\end{theorem}
\begin{proof} %(Prop.~\ref{conv:spider})
We again have
\begin{align*}
\Phi(x_{k+1}) &\overset{(1)}{\leq} \Phi(x_k) + \eta_k \ip{\grad \; \Phi(x_k)}{\Exp_{x_k}^{-1}(y_k)} + \frac{1}{2} M_{\Phi} \eta_k^2 \\
&\overset{(2)}{\leq} \Phi(x_k) + \eta_k \ip{g_k(x_k)}{\Exp_{x_k}^{-1}(y_k)} + \eta_k \ip{\grad \; \Phi(x_k) - g_k(x_k)}{\Exp_{x_k}^{-1}(y_k)} + \frac{1}{2} M_{\Phi} \eta_k^2 \; ,
\end{align*}
where, (1) follows from Lemma~\ref{A.smooth} and (2) from ``adding a zero" with respect to $g_k$. We again apply the Cauchy-Schwartz inequality to the inner product and make use of the fact that the geodesic distance between points in $\Xc$ is bounded by its diameter:
\begin{align}
\ip{\grad \; \Phi(x_k) - g_k(x_k)}{\Exp_{x_k}^{-1}(y_k)} \leq \norm{\grad \; \Phi(x_k) - g_k(x_k)} \cdot \underbrace{\norm{\Exp_{x_k}^{-1}(y_k)}}_{\leq {\rm diam}(\Xc)} \; .
\end{align}
This gives (with $D := {\rm diam}(\Xc)$)
\begin{align*}
\Phi(x_{k+1}) &\leq \Phi(x_k) + \eta_k \ip{g_k(x_k)}{\Exp_{x_k}^{-1}(y_k)} + \eta_k D \norm{\grad \; \Phi(x_k) - g_k(x_k)} + \frac{1}{2} M_{\Phi} \eta_k^2 \; .
\end{align*}
%We know subtract $\Phi(x^*)$ from both sides and introduce the shorthand $\Delta_k = \Phi(X_k) - \Phi(x^*)$ for the optimality gap. 
Taking expectations, we get
\begin{align*}
\mathbb{E}\left[ \Phi(x_{k+1}) \right] &\leq \mathbb{E}\left[ \Phi(x_{k}) \right] + \eta_k  \underbrace{\mathbb{E}\left[ \ip{g_k(x_k)}{\Exp_{x_k}^{-1}(y_k)} \right]}_{=-\mathbb{E}\left[\hat{\mathcal{G}}(x_k)\right]}  + \eta_k D \underbrace{\mathbb{E}\left[ \norm{\grad \; \Phi(x_k) - g_k(x_k)} \right]}_{\leq \epsilon} + \frac{1}{2} M_{\Phi} \eta_k^2 \; .
\end{align*}
With Lemma~\ref{lem.spider} and the definition of the stochastic Frank-Wolfe gap, this can be rewritten as 
\begin{align*}
\mathbb{E}\left[ \Phi(x_{k+1}) \right] &\leq \mathbb{E} \left[ \Phi(x_{k}) \right] - \eta_k \mathbb{E}\left[ \hat{\mathcal{G}}(x_k) \right] + \eta_k D \epsilon + \frac{1}{2} M_{\Phi} \eta_k^2 \; .
\end{align*}
Summing and telescoping gives 
\begin{align*}
\mathbb{E} \left[\mathcal{G}(\hat{x})\right] \sum_k \eta_k &\leq \mathbb{E} \left[ \Phi(x_{0}) \right] - \mathbb{E} \left[ \Phi(x_{K}) \right] +  D \epsilon \sum_k \eta_k + \frac{1}{2} M_{\Phi} \sum_k \eta_k^2 \\
&\leq  \left( \Phi (x_{0}) - \mathbb{E}\left[ \Phi(x_K) \right]\right)  +  D \epsilon \sum_k \eta_k + \frac{1}{2} M_{\Phi} \sum_k \eta_k^2 \; ,
\end{align*}
where we have again used the definition of the output in Algorithm~\ref{alg.spider}; in particular, that $\mathbb{E} \left[\mathbb{E} \left[\hat{\mathcal{G}}(x_K)\right] \right] = \mathbb{E} \left[\mathcal{G}(\hat{x})\right]$. With $\eta_k=\eta=\frac{1}{\sqrt{K}}$, this becomes 
\begin{align*}
\underbrace{K \eta}_{= \sqrt{K}} \mathbb{E} \left[\mathcal{G}(\hat{x})\right] &\leq \left( \Phi (x_{0}) - \mathbb{E}\left[ \Phi(x_K) \right]\right) +  D \epsilon \underbrace{K \eta}_{= \sqrt{K}} + \frac{1}{2} M_{\Phi} \underbrace{K \eta^2}_{= 1} \; .
\end{align*}
Note, that $\epsilon=\frac{1}{n}=\frac{1}{\sqrt{K}}$. Dividing by $\sqrt{K}$ then gives the claim:
\begin{align}
\mathbb{E} \left[\mathcal{G}(\hat{x})\right] &\leq \frac{1}{\sqrt{K}} \left(C_{x_0}  +  D \underbrace{\epsilon \sqrt{K}}_{=1} + \frac{1}{2} M_{\Phi} \right)\; ,
\end{align}
where $C_{x_0} > \Phi(x_0) - \Phi(x^\star)$ depends on the initialization only and $x^{\star}$ is a first-order stationary point.
\end{proof}
\begin{cor}
\spiderfw obtains an $\epsilon$-accurate solution with SFO/ IFO complexity of $O \left(\frac{1}{\epsilon^3}	\right)$ and RLO complexity of $O \left(\frac{1}{\epsilon^2}	\right)$.
\end{cor}
\begin{proof}
It follows directly from Theorem~\ref{conv:spider} that \spiderfw has an RLO complexity of $O \left(\frac{1}{\epsilon^2}	\right)$.  For the SFO complexity, consider a stochastic objective $\Phi$. Then
\begin{align*}
SFO = \sum_{s=1}^n \left( \vert S_1 \vert + \mathbb{E} \left[ \sum_{k=2}^n \vert S_2 \vert \right] \right) \; .
\end{align*}
We have
\begin{align*}
\mathbb{E} \left[ \sum_{k=2}^n \vert S_2 \vert \right] 
= \mathbb{E}  \left[\sum_{k=2}^n \frac{2nL \norm{\Exp_{x_{k-1}}^{-1}(x_k)}}{\epsilon^2} \right]
\overset{(1)}{\lesssim}  \frac{2n^2 L^2 (D^2 \eta^2)}{2 \epsilon^2}  \overset{(2)}{=} O \left(\frac{1}{\epsilon^2} \right) \; ,
\end{align*}
where (1) follows from $\norm{\Exp_{x_{k-1}}^{-1}(x_k)} \leq \eta D$ (see Equation~\ref{eq:diam-eta}) and (2) from $\eta = \frac{1}{n}$ by construction. This gives
 \begin{align*}
SFO = O \left( n \left(\frac{1}{\epsilon^2} + \frac{1}{\epsilon^2}		\right) \right) = O \left( \frac{1}{\epsilon^3} \right).
\end{align*}
An analogous argument gives the IFO complexity, if $\Phi$ has a finite-sum structure.
\end{proof}

\noindent We again consider the special case of g-convex objectives for completeness. Here, we obtain a result on function suboptimality:
\begin{cor}
If $\Phi$ is g-convex, one can show under the assumptions of Theorem~\ref{conv:spider} a similar convergence rate for the optimality gap, i.e., $\mathbb{E} \left[\Delta_{k} \right] = O(1/\sqrt{K})$.
\end{cor}
\noindent The proof is analogous to the proof of Corollary~\ref{cor:srfw} (for stochastic objectives) and of Corollary~\ref{cor:svrfw} (for objectives with finite-sum structure).

%%%%%%%%%%%%%%%%
\section{Experiments}
\vspace*{-8pt}
\label{sec:experiments}
(Stochastic) Riemannian optimization is frequently considered in the machine learning literature, including for the computation of hyperbolic embeddings~\citep{sala}, low-rank matrix and tensor factorization~\citep{vandereycken} and eigenvector based methods~\citep{journee,zhang16,tripuraneni2018averaging}.

In this section we validate the proposed stochastic algorithms by comparison with the deterministic \rfw~\citep{weber-sra} and state-of-the-art stochastic Riemannian optimization methods. All experiments were performed in \matlab.

Our numerical experiments use synthetic data, consisting of sets of 
symmetric, positive definite matrices. We generate 
matrices by sampling real matrices of dimension $d$ uniformly at random $M_i \sim \mathcal{U}(\mathbb{R}^{d \times d})$
and then multiplying each with its transpose $M_i \gets M_i M_i^T$.  To generate ill-conditioned matrices, we sample matrices with a rank deficit $U_i \sim \mathcal{U}(\mathbb{R}^{d \times d})$ (with ${\rm rank}(U)<d$) and set $B_i \gets \delta I + U_i U_i^T$ (for a small $\delta >0$).

Throughout the experiments, the hyperparameter choices $(b,K)$ are guided by the specifications in Algorithm~\ref{alg.srfw} (for \srfw) and Algorithm~\ref{alg.svrfw} (for \svrfw) and their theoretical analysis. All \rfw methods are implemented with decreasing step sizes.

\subsection{Riemannian centroid}
\begin{figure*}[ht]
\begin{center}
\includegraphics[width=\textwidth]{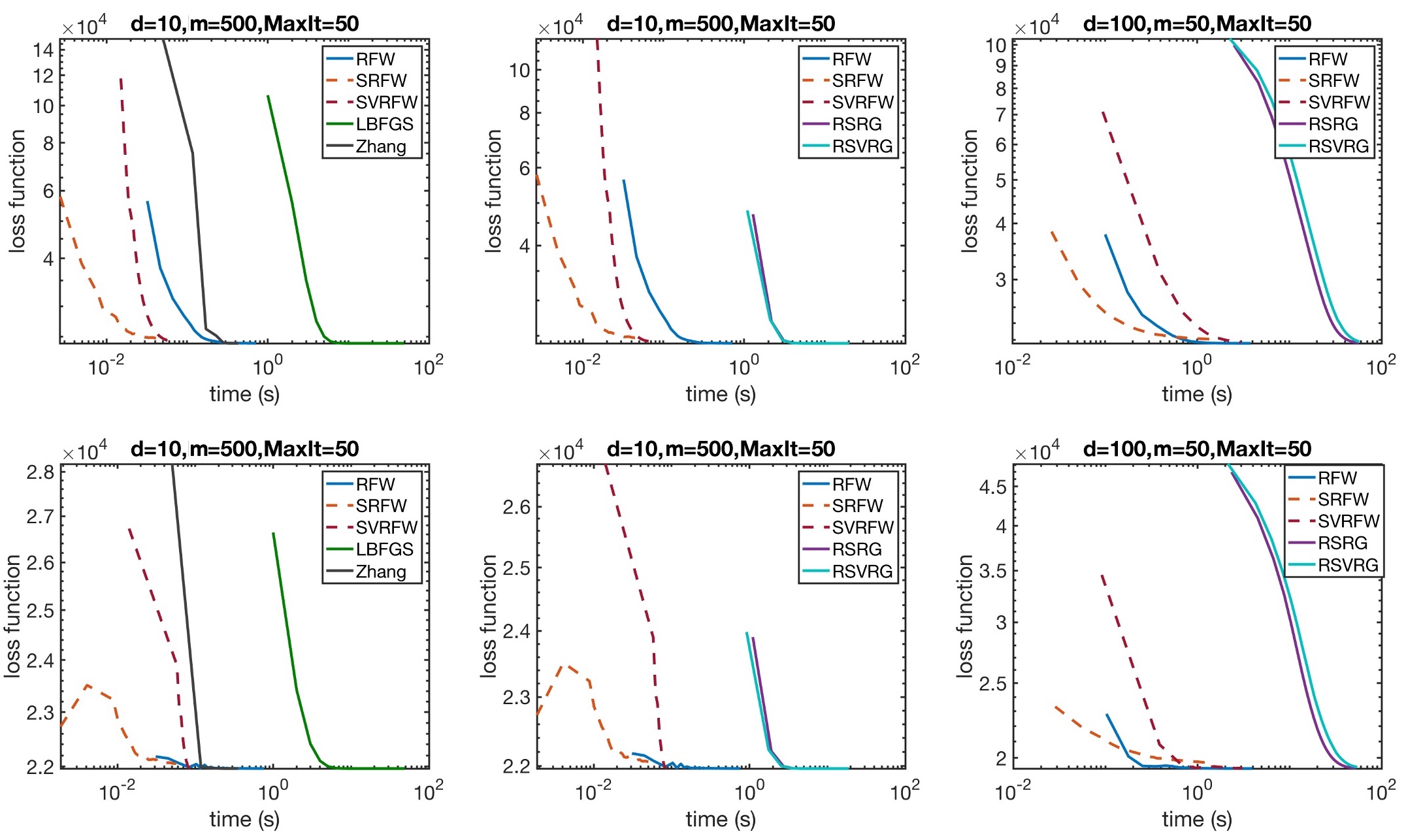}
\caption{\textbf{Riemannian centroid.} \rfw and its stochastic variants in comparison with state-of-the-art Riemannian optimization methods (\emph{parameters:} $d$ - size of matrices, $m$ - number of matrices, $MaxIt$ - number of iterations). All experiments are initialized with the harmonic mean.  Hereby, we compare against deterministic algorithms (\textsc{LBFGS} and \textsc{Zhang}, left) as well as recent state-of-the-art stochastic Riemannian algorithms \textsc{R-SRG} and \textsc{RSVRG} (middle and right). The results in the top row are for well-conditioned, the results in the bottom row for ill-conditioned matrices.}
	\label{fig:srfw-karcher}
\end{center}
\end{figure*}
The computation of the \emph{Riemannian centroid} (also known as the \emph{geometric matrix mean} or the \emph{Karcher mean}) is a canonical benchmark task for testing Riemannian optimization methods~\citep{zhang16,kasai-mishra,ropt-package}.  %It frequently appears as subroutine in statistical inference tasks that require the computation of averages over matrix-valued inputs.
Besides its importance as a benchmark, the Karcher mean is a fundamental subroutine in many machine learning methods, for instance, in the computation of hyperbolic embeddings~\citep{sala}. Although the Karcher mean problem is nonconvex in Euclidean space, it is g-convex in the Riemannian setting. This allows for the application of \rfw, in addition to the stochastic methods discussed above. \rfw requires the computation of the full gradient in each iteration step, whereas the stochastic variants implement gradient estimates at a significantly reduced computational cost. This results in observable performance gains as shown in our experiments (Fig.~\ref{fig:srfw-karcher}).

Formally, the Riemannian centroid is defined as the mean of a set $M = \lbrace M_i\rbrace$ of $d\times d$ positive definite matrices (we write $\vert M \vert = m$) with respect to the Riemannian metric. This task requires solving
\begin{align*}
\min_{H \preceq X \preceq A} \; &\sum_{i=1}^m w_i \delta_R^2(X,M_i ) = \sum_{i=1}^m w_i \norm{\log \left(	X^{-1/2} M_i 	X^{-1/2} \right)}_F^2 \; ,
\end{align*}
where $\norm{\cdot}_F$ denotes the Frobenius norm. The well-known \emph{matrix means inequality} bounds the Riemannian mean from above and below with respect to the Löwner order: The \emph{harmonic mean} $H :=\left(	\sum_i	w_i M_i^{-1}\right)^{-1}$ gives a lower bound on the geometric matrix mean, while the arithmetic mean $A := \sum_i w_i M_i$ provides an upper bound~\citep{bhatia07}.  
This allows for phrasing the computation of the Riemannian centroid as a constrained optimization task with interval constraints given by the harmonic and arithmetic means (though it could be solved as unconstrained task too). 
Writing $\phi_i(X) = w_i\delta_R^2(X,M_i)$, we note that the gradient of the objective is given by $\nabla\phi_i(X) = w_i\inv{X}\log(XM_i^{-1})$ (see e.g., \cite[Ch.6]{bhatia07}), whereby the corresponding Riemannian ``linear'' oracle reduces to solving
\begin{align}
  \label{eq:4}
  Z_k \gets \argmin_{H \preceq Z \preceq A} \ip{X_k^{1/2}\nabla \phi_i(X_k)X_k^{1/2}}{\log(X_k^{-1/2}ZX_k^{-1/2})} \; .
\end{align}
Remarkably, \eqref{eq:4} can be solved in closed form~\citep[Theorem 4.1]{weber-sra}, which we exploit to achieve an efficient implementation of \rfw and \textsc{Stochastic} \rfw.
For completeness, we recall the theorem below:
\begin{theorem}[Theorem 4.1~\citep{weber-sra}]
  \label{thm:logtrace}
  Let $L, U \in \pd_d$ such that $L \prec U$. Let $S \in \H_d$ and $X \in \pd_d$ be arbitrary. Then, the solution to the optimization problem
  \begin{equation}
    \label{eq.12.R}
    \min_{L \preceq Z \preceq U}\quad \trace(S \log(XZX)),
  \end{equation}
  is given by $Z = X^{-1}Q \left( P^* [-\sgn(D)]_+ P  + \hat{L}\right) Q^*X^{-1}$, where $S=QDQ^*$ is a diagonalization of $S$, $\hat{U} - \hat{L}=P^* P$ with $\hat{L}=Q^* X L X Q$ and $\hat{U}=Q^* X U X Q$.
\end{theorem}
\noindent Setting $L=H$, $U=A$ and $S=X_k^{1/2}\nabla \phi_i(X_k)X_k^{1/2}$,  this result gives a closed form solution to Eq.~\ref{eq:4}.

To evaluate the efficiency of our methods, we compare against state-of-the-art algorithms. First,  Riemannian \textsc{LBFGS}, a quasi-Newton method~\citep{lbfgs}, for which we use an improved limited-memory version of the method available in \emph{Manopt}~\citep{manopt}. Secondly \textsc{Zhang}'s method~\citep{zhang}, a recently published majorization-minimization method for computing the geometric matrix mean. Against both (deterministic) algorithms we observe significant performance gains (Fig.~\ref{fig:srfw-karcher}). In~\citep{weber-sra}, \rfw is compared with a wide range of Riemannian optimization methods and varying choices of hyperparameters. In those experiments, \textsc{LBFGS} and \textsc{Zhang}'s method were reported to be especially competitive, which motivates our choice. We further present two instances of comparing \textsc{Stochastic} \rfw against stochastic gradient-based methods (\textsc{RSrg} and \textsc{Rsvrg}~\citep{ropt-package}), both of which are outperformed by our \rfw approach.

In a second experiment, we compare the accuracy (i.e., $\frac{\vert \phi(x_{\rm final}) - \phi(x^*) \vert}{\vert \phi(x^*) \vert}$) of \rfw and its stochastic variants with that of \textsc{RSrg} and \textsc{Rsvrg}. Figure~\ref{fig:accuracy} shows that \textsc{Stochastic} \rfw reach a medium accuracy fast; however, ultimately \rfw, as wells as \textsc{R-Srg} and \textsc{Rsvrg} reach a higher accuracy.  \textsc{Stochastic} \rfw is therefore particularly suitable for data science and machine learning applications, where we encounter high-dimensional, large-scale data sets and very high accuracy is not required.
\begin{figure*}[ht]
\begin{center}
\includegraphics[width=\textwidth]{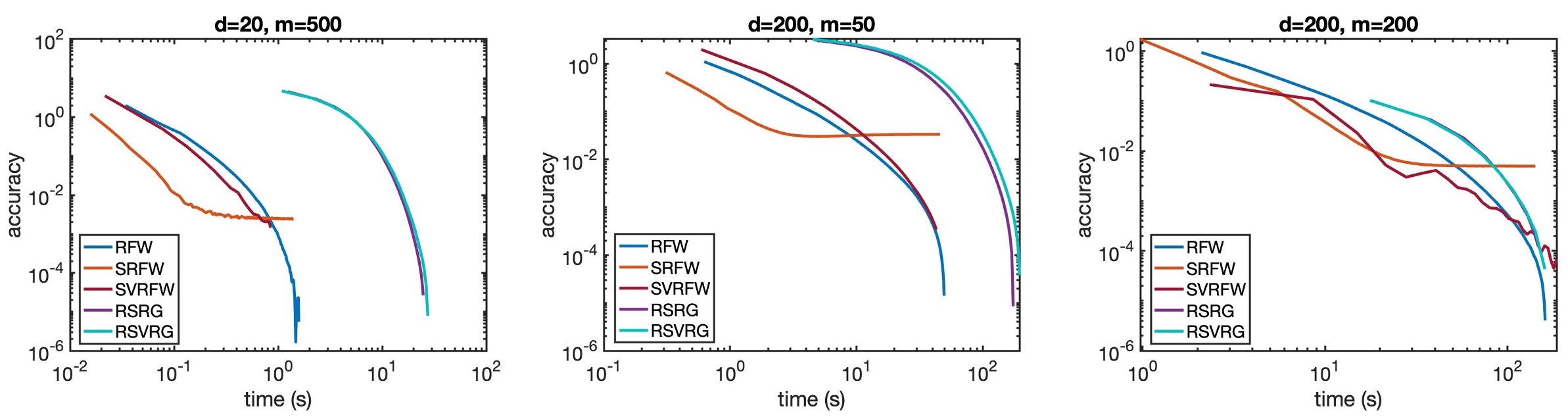}
\caption{\textbf{Riemannian centroids.} Accuracy of \rfw and stochastic variants in comparison with \textsc{RSrg} and \textsc{Rsvrg} for inputs of different size ($d$: size of matrices, $m$: number of matrices).  All experiments are initialized with the arithmetic mean. }
	\label{fig:accuracy}
\end{center}
\end{figure*}

We note that the comparison experiments are not quite fair to our methods, as neither \textsc{R-Srg} nor \textsc{Rsvrg} implement the noted projection operation (see discussion in section~\ref{sec:proj-free}) required to align their implementation with their theory.

% . They found that \rfw outperforms state-of-the-art methods. Among the tested algorithms
\subsection{Wasserstein Barycenters}
\begin{figure*}[ht]
\begin{center}
\includegraphics[width=\textwidth]{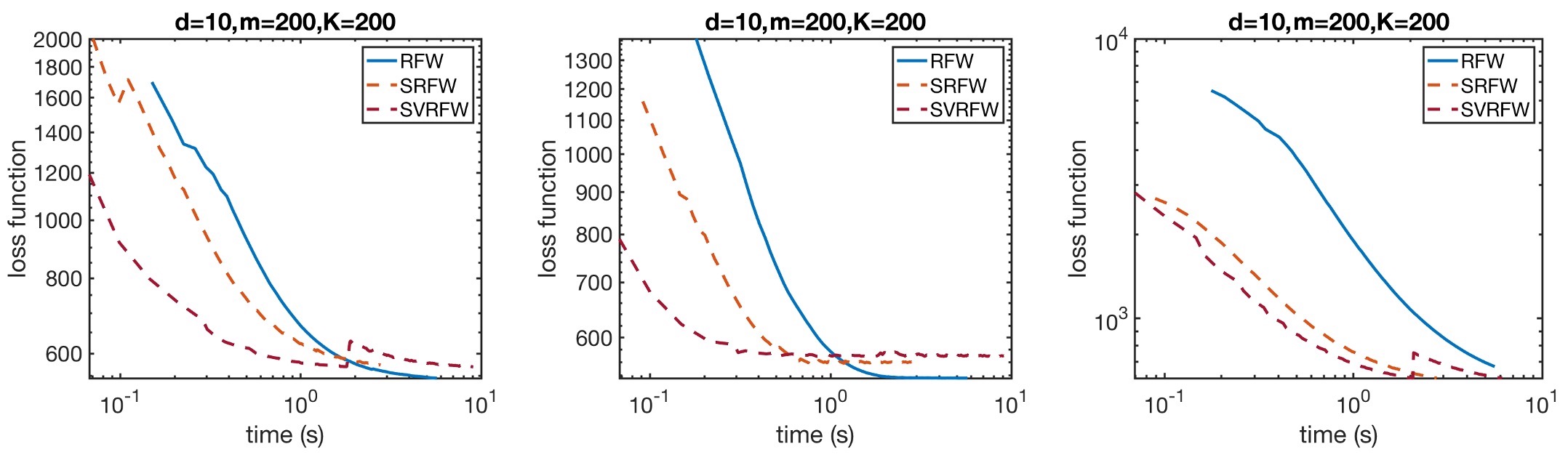}
\caption{\textbf{Wasserstein barycenters.} Performance of \rfw and stochastic variants for well-conditioned inputs of fixed size ($d$: size of matrices, $m$: number of matrices, $K$: number of iterations) with different initializations: $X_0 \sim \mathcal{C}$ (left), $X_0=\frac{1}{2}\left(\alpha I + A\right)$ (middle) and $X_0=A$ (right).  Here,  $A$ denotes the arithmetic mean of $\mathcal{C}$ and $\alpha$ the smallest eigenvalue over $\mathcal{C}$.}
	\label{fig:srfw-wm}
\end{center}
\end{figure*}
\begin{figure*}[t]
\begin{center}
\includegraphics[width=\textwidth]{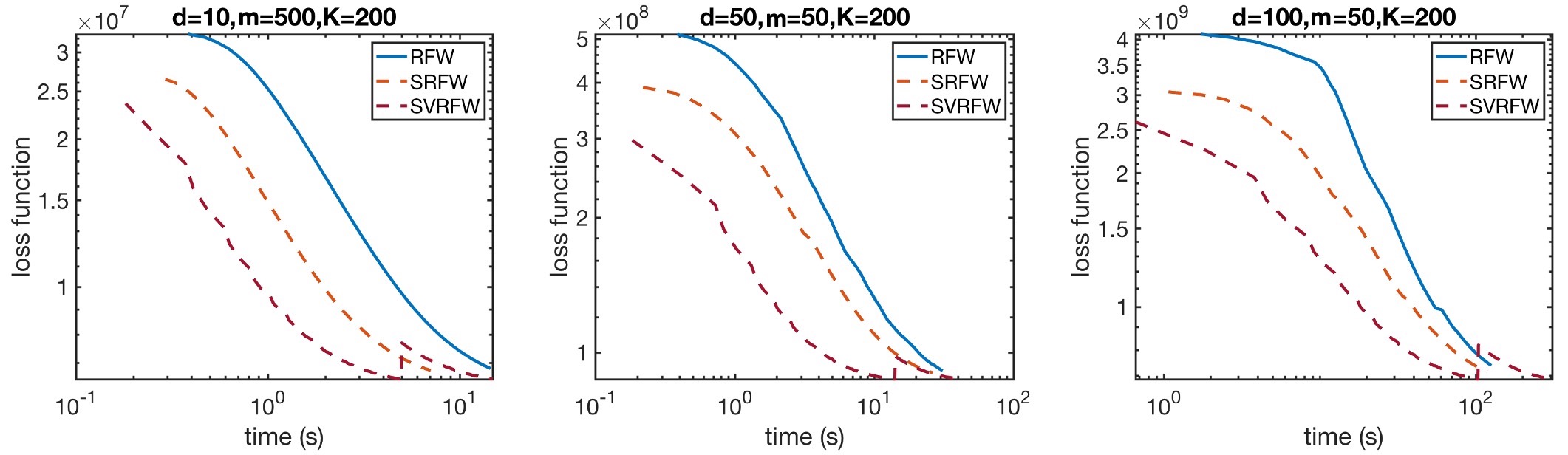}
\caption{\textbf{Wasserstein barycenters for MVNs.} Performance of \rfw and stochastic variants for well-conditioned inputs of different sizes ($d$: size of matrices, $m$: number of matrices, $K$: number of iterations); initialized at $X_0=A$.  Again,  $A$ denotes the arithmetic mean of $\mathcal{C}$ and $\alpha$ the smallest eigenvalue over $\mathcal{C}$.}
	\label{fig:srfw-wm}
\end{center}
\end{figure*}
%
%\subsubsection{Multivariate Gaussians}
The computation of means of empirical probability measures with respect to the optimal transport metric (or \emph{Wasserstein distance}) is a basic task in statistics. Here, we consider the problem of computing such \emph{Wasserstein barycenters} of multivariate (centered) Gaussians.
This corresponds to the following minimization task on the \emph{Gaussian density manifold} (also known as \emph{Bures manifold}):\footnote{Interestingly, this problem turns out to be Euclidean convex (more precisely,  a nonlinear semidefinite program). However, a Riemannian approach exploits the problem structure more explicitely. }
\begin{equation}\label{eq:obj-wm}
\min_{\alpha I \preceq X \preceq A} \; \sum_{i=1}^M d_W^2(X,\mathcal{C}) 
 = \sum_i w_i \bigl[	\trace(C_i+X) - 2 \trace \bigl(C_i^{1/2} X C_i^{1/2} \bigr)^{1/2}		\bigr]  \; ,
\end{equation}
\noindent where $\mathcal{C} = \lbrace C_i \rbrace \subseteq \mathbb{P}(n), \; \vert \mathcal{C} \vert = m$ are the covariance matrices of the Gaussians and $\alpha$ denotes their minimal eigenvalue over $\mathcal{C}$. 
Note that the Gaussian density manifold is isomorphic to the manifold of symmetric positive definite matrices considered in the previous section. This allows for a direct application of \rfw to Eq.~\ref{eq:obj-wm}, albeit with a different set of constraints.

A closely related problem is the task of computing Wasserstein barycenters of \emph{matrix-variate Gaussians}, i.e., multivariate Gaussians whose covariance matrices are expressed as suitable Kronecker products. Such models are of interest in several inference problems, see for instance~\citep{MVN}. By plugging in Kronecker structured covariances into~\eqref{eq:obj-wm}, the corresponding barycenter problem takes the form
\begin{equation}\label{eq:matrix-variate}
\min_{X \succ 0} \sum_{i=1}^n \trace (A_i \otimes A_i) + \trace (X \otimes X) 
- 2 \trace \bigl[	(A_i \otimes A_i)^{1/2} (X \otimes X) (A_i \otimes A_i)^{1/2}	\bigr]^{1/2} \; .
\end{equation}
Remarkably, despite the product terms, problem~\eqref{eq:matrix-variate} turns out to be (Euclidean) convex (Lemma~\ref{lem:mvn-convex}). This allows one to apply (g-) convex optimization tools, and use convexity to conclude global optimality. This result should be of independent interest.
\begin{lem}\label{lem:mvn-convex}
The barycenter problem for matrix-variate Gaussians (Eq.~\ref{eq:matrix-variate}) is convex.
\end{lem}
\noindent For the proof, recall the following well-known properties of Kronecker products:
\begin{lem}[Properties of Kronecker products]\label{lem:prop-kronecker}
Let $A,B,C,D \in \mathbb{P}^d$.
\begin{enumerate}
\item $(A \otimes A)^{1/2} = A^{1/2} \otimes A^{1/2}$;
\item $AC \otimes BD = (A \otimes B)(C \otimes D)$.
\end{enumerate}
\end{lem}
\noindent Furthermore, recall the Ando-Lieb theorem~\citep{ando}:
\begin{theorem}[Ando-Lieb]
Let $A,B \in \mathbb{P}^d$. Then the map $(A,B) \mapsto A^\gamma \otimes B^{1-\gamma}$ is jointly concave for $0 < \gamma < 1$.
\end{theorem}
\noindent Equipped with those two arguments, we can prove the lemma.\\
\begin{proof}(Lemma~\ref{lem:mvn-convex})
First, note that 
\begin{align*}
\trace (A_i \otimes A_i) &= (\trace A_i) (\trace A_i) = (\trace A_i)^2 \qquad \forall \; i=1, \dots n \\
\trace (X \otimes X) &= (\trace X) (\trace X) = (\trace X)^2 \; .
\end{align*}
Next, consider the third term. We have 
\begin{align*}
\trace \left[ \left(	(A_i \otimes A_i)^{1/2} (X \otimes X)	(A_i \otimes A_i)^{1/2}	\right)^{1/2} \right]
&\overset{(1)}{=} \trace \left[ \left(	(A_i^{1/2} X \otimes A_i^{1/2} X)	(A_i \otimes A_i)^{1/2}	\right)^{1/2} \right] \\
&\overset{(1)}{=} \trace \left[ \left(	(A_i^{1/2} X A_i^{1/2}) \otimes (A_i^{1/2} X A_i^{1/2})	\right)^{1/2} \right] \\
&\overset{(2)}{=} \trace \left[ \left(A_i^{1/2} X A_i^{1/2}\right)^{1/2} \otimes \left( A_i^{1/2} X A_i^{1/2}	\right)^{1/2} \right] \; ,
\end{align*}
where (1) follows from Lemma~\ref{lem:prop-kronecker}(ii) and (2) from Lemma~\ref{lem:prop-kronecker}(i). Note that $X \mapsto A^{1/2} X A^{1/2}$ is a linear map. Therefore, we can now apply the Ando-Lieb theorem with $\gamma=\frac{1}{2}$, which establishes the concavity of the trace term. Its negative is convex and consequently, the objective is a sum of convex functions. The claim follows from the convexity of sums of convex functions.
\end{proof}

One can show that the Wasserstein mean is upper bounded by the arithmetic mean $A$ and lower bounded by $\alpha I$, where $\alpha$ denotes the smallest eigenvalue over $\mathcal{C}$~\citep{bhatia,bhatia2}. This allows for computing the Wasserstein mean via constrained optimization (though, again, one could use unconstrained tools too).
For computing the gradient, note that the Riemannian gradient $\grad \phi(X)$ can be written as
$\grad \phi(X) = X \nabla \phi(X) - \nabla \phi(X) X$, 
where $\nabla \phi$ is the Euclidean gradient (where $\phi$ denotes the objective in~\eqref{eq:obj-wm}). It is easy to show, that
\begin{align*}
\nabla \phi(X) = \sum_i w_i \left( I - \left( C_i X	\right)^{-1/2} C_i	\right) \; ,
\end{align*}
which directly gives the gradient of the objective.

We evaluate the performance of our stochastic \rfw methods against the deterministic \rfw method for different initializations (Fig.~\ref{fig:srfw-wm}). Our results indicate that all three initializations are suitable. This suggests, that (stochastic) \rfw is not sensitive to initialization and performs well even if not initialized close to the optimum. In a second experiment, we compute Wasserstein barycenters of MVNs for different input sizes. Both experiments indicates that especially the purely stochastic \srfw improves on \rfw with comparable accuracy and stability. We did not compare against projection-based methods in the case of Wasserstein barycenters, since to our knowledge there are no implementations with the appropriate projections available. 

%%%%%%%%%%%%%%%%
\section{Discussion}
We introduced three stochastic Riemannian Frank-Wolfe methods, which go well-beyond the deterministic \rfw algorithm proposed in~\citep{weber-sra}. In particular, we (i) allow for an application to nonconvex, stochastic problems; and (ii) improve the oracle complexities by replacing the computation of full gradients with stochastic gradient estimates. For the latter task, we analyze both fully stochastic and semi-stochastic variance-reduced estimators. Moreover, we implement the recently proposed \spider technique that significantly improves the classical Robbins-Monroe and variance-reduced gradient estimates by circumventing the need to recompute full gradients periodically.

We discuss applications of our methods to the computation of the Riemannian centroid and Wasserstein barycenters, both fundamental subroutines of potential value in several applications, including in machine learning. In validation experiments, we observe performance gains compared to the deterministic \rfw as well as state-of-the-art deterministic and stochastic Riemannian methods.

This paper focused on developing a non-asymptotic convergence analysis and on establishing theoretical guarantees for our methods. Future work includes implementation of our algorithms for other manifolds and other classical Riemannian optimization tasks (see, e.g.,~\citep{absil_review}). This includes tasks with constraints on determinants or condition numbers. An important example for the latter is the task of learning a DPP kernel (see, e.g.,~\citep{mariet15}), which can be formulated as a stochastic, geodesically convex problem. We hope to explore practical applications of our approach to large-scale constrained problems in machine learning and statistics. 

Furthermore, instead of using exponential maps, one can reformulate our proposed methods using retractions. For projected-gradient methods, the practicality of retraction-based approaches has been established~\citep{absil_book}, rendering this a promising extension for future research.

%%%%%%%%%%%%%%%%%%
\bibliographystyle{plainnat}
\bibliography{ref}

\end{document}